\documentclass[final]{elsarticle}
\usepackage{latexsym, graphicx, epsfig, amsmath, amsfonts,amssymb,subfigure,algorithm,algorithmicx,booktabs,bm}
 \def\d{\delta}

\def\x{\xi}
\def\Pb{\textrm{Prob}}
\def\mb{\mathbf}
\def\z{\zeta}

\def\i{\mathbf{i}}
\def\j{\mathbf{j}}
\def\k{\mathbf{k}}

\def\t{\tilde}

\def\E{{\mathbb E}}

\def\P{{\mathcal P}}

\def\O{{\Omega}}
\def\o{{\omega}}
\def\A{{\mathcal{A}}}

\def\_#1{{\underline{#1}}}
\def\Re{\mathop{\rm Re}\nolimits}

\newtheorem{theorem}{Theorem}
\newproof{proof}{Proof}
\begin{document}
\begin{frontmatter}

\title{Mesh refinement for uncertainty quantification through model reduction}

\author{Jing Li}

\address{School of Mathematics, University of Minnesota,
          Minneapolis, MN 55455. Email: lixxx873@umn.edu.}

\author{Panos Stinis}

\address{School of Mathematics, University of Minnesota,
          Minneapolis, MN 55455. Email: stinis@umn.edu.}

\begin{abstract}
We present a novel way of deciding when and where to refine a mesh in probability space in order to facilitate the uncertainty quantification in the presence of discontinuities in random space. A discontinuity in random space makes the application of generalized polynomial chaos expansion techniques prohibitively expensive. The reason is that for discontinuous problems, the expansion converges very slowly. An alternative to using higher terms in the expansion is to divide the random space in smaller elements where a lower degree polynomial is adequate to describe the randomness. In general, the partition of the random space is a dynamic process since some areas of the random space, particularly around the discontinuity, need more refinement than others as time evolves. In the current work we propose a way to decide when and where to refine the random space mesh based on the use of a reduced model. The idea is that a good reduced model can monitor accurately, within a random space element, the cascade of activity to higher degree terms in the chaos expansion. In terms, this facilitates the efficient allocation of computational sources to the areas of random space where they are more needed. For the Kraichnan-Orszag system, the prototypical system to study discontinuities in random space, we present theoretical results which show why the proposed method is sound and  numerical results which corroborate the theory.
\end{abstract}

\end{frontmatter}

% main text
\section{Introduction} \label{sec:intro}
Generalized polynomial chaos (gPC) is a frequently used approach to represent non-statistical uncertain quantities when solving differential equations involving uncertainty in initial conditions, boundary conditions, randomness in material parameters and etc. Based on the results of Wiener\cite{Wiener38}, spectral expansion employing Hermite orthogonal polynomials was introduced by Ghanem et. al \cite{GhanemS91} for various uncertainty quantification problems in mechanics. This method was generalized by Xiu and Karniadakis \cite{XiuK_SISC02,Xiu_CICP09} to include other families of orthogonal polynomials. When the solution is sufficiently regular with respect to the random inputs, the gPC expansion has an exponential convergence rate \cite{XiuK_SISC02}. However, if the solution is not smooth, the rate of convergence of gPC deteriorates similarly to the deterioration of a Fourier expansion of non-smooth functions \cite{HesthavenGG}. The reason for the lack of smoothness can be, for example, the presence of certain values of the random input around which the solution may change qualitatively (this is called a discontinuity in random space). For such problems, the brute force approach of using more terms in the gPC expansion is prohibitively expensive. 

An alternative to using higher terms in the expansion is to divide the random space in smaller elements where a lower degree polynomial is adequate to describe the randomness \cite{WanK_JCP05}. This approach requires a criterion (a mechanism) to decide how to best partition the random space. Ideally, the criterion will focus on parts of random space, like discontinuities, where there is more sensitive dependence on the value of the random parameters.  In addition to the presence of discontinuities in random space, there are problems which simply have too many sources of uncertainty to allow for a high degree expansion in all dimensions of random space. For some problems not all of the sources of uncertainty are equally important. This means that some directions in random space need more refinement than others. So, one needs to be able to identify correctly these directions and allocate accordingly the available computational resources. 

In \cite{Stinis09}, one of the current authors proposed a novel algorithm for performing mesh refinement in physical space by using a reduced model. The algorithm is based on the observation that the need for mesh refinement is dictated by the cascade of activity to scales smaller than the ones resolved (depending on the physical context this could mean a mass or an energy cascade). A good reduced model should be able to effect with accuracy the necessary transfer of activity across scales. Thus, a good reduced model can be used to decide when to refine. 

What is needed to define a mesh refinement algorithm is a criterion to determine whether it is time to perform mesh refinement. In \cite{Stinis09}, this criterion was based on monitoring the rate of change of the $L_2$ norm of the solution at the resolved scales as computed by the reduced model (note that the $L_2$ norm corresponds to the mass or energy in many physical contexts). When this rate of change exceeds a prescribed tolerance the algorithm performs mesh refinement. The suitability of the rate of change of the $L_2$ norm as an indicator for the need to refine is shown in Appendix B. In particular, we show that the expression for the rate of change of the $L_2$ norm for the resolved scales has the same functional form as the expression for the rate of change of the $L_2$ error of the reduced model. Thus, by keeping, through mesh refinement, the rate of change of the $L_2$ norm for the resolved scales under a prescribed tolerance, we can keep the error of the calculation under control (see Section 3 for more details). 

The paper is organized as follows. In Section 2, we recall the framework for the stochastic Galerkin formulation of a random system. The proposed mesh refinement algorithm is presented in Section 3.  Section 4 contains numerical results from the application of the algorithm. Conclusions are drawn in Section 6. Finally, Appendix A contains the Galerkin formulation of the Kraichnan-Orszag system as well as the reduced model used in the mesh refinement algorithm. Appendix B contains a proof of convergence of the reduced model.    

%%%%%%%%%%%%%%%%%%End of Introduction%%%%%%%%%%%%%%%%%%%

\section{gPC representation of uncertainty} \label{sec:gPC}
Let $(\O,\A,\P)$ be a probability space, where $\O$ is the event space and $\P$ is the probability measure defined on the $\sigma-$ algebra of subsets of $\O$. Let $\bm{\x} = (\x_1,\cdots,\x_d)$ be a $d$-dimensional random vector for the random event $\o\in\O$. Without loss of generality, consider an orthonormal generalized polynomial chaos basis $\{\Phi_\i\}_{|\i|=0}^{\infty}$ spanning the space of second-order random processes on this probability space ($\i = (i_1,\cdots,i_d)\in \mathbb{N}^d_0$ is a multi-index with $|\i|=i_1+\cdots+i_d.$) The basis functions $\Phi_\i(\bm{\x}(\o))$ are polynomials of degree $|\i|$ with orthonormal relation
\begin{equation}\label{orth_basis}
\langle\Phi_\i,\Phi_{\j}\rangle = \delta_{\i\j},
\end{equation}
where $\delta_{\i\j}$ is the Kronecker delta and the inner product between two functions $f(\bm{\x})$ and $g(\bm{\x})$ is defined by
\begin{equation}\label{inner}
\langle f(\bm{\x}),g(\bm{\x})\rangle=\int_{\O}f(\bm{\x})g(\bm{\x})d\P(\bm{\x}).
\end{equation}
A general second-order random process $u(\o)\in L_2(\O,\A,\P)$ can be expressed by gPC as
\begin{equation}\label{gPC_expan}
u(\o) = \sum_{|\i|=0}^\infty u_\i\Phi_{\i}(\bm{\x}(\o)),
\end{equation}
The mean and variance of $u(\o)$ can be expressed independently of the choice of basis as
\begin{equation}\label{mu_and_var}
\E(u(\bm{\x})) = u_\mathbf{0}, \quad \text{Var}(u(\bm{\x})) = \sum_{|\i|=0}^\infty u_{\i}^2,
\end{equation}
respectively.
For numerical implementation, \eqref{gPC_expan} is truncated to a finite number of $n$ terms, and we set
\begin{equation}\label{gPC_expan_fi}
u(\o) = \sum_{|\i|=0}^p u_\i\Phi_{\i}(\bm{\x}(\o)),
\end{equation}
where $p$ is the highest order of the polynomial bases and $n = \left(\begin{array}{c}p\\d+p\end{array}\right)$.

\subsection{gPC Galerkin method for stochastic differential equations}
Consider the following stochastic differential equation
\begin{equation}\label{sto_eq}
u_t(\mathbf{x},t;\o) = \mathcal{L}(\mathbf{x},t,\o;u),
\end{equation}
where $u:=u(\mathbf{x},t;\o)$ is the solution. Operator $\mathcal{L}$ usually involves differentiations in space and can be nonlinear. Appropriate initial conditions and boundary conditions sometimes involving random parameters are assumed. The solution $u$ can be approximated by the truncated gPC expansion
\begin{equation}\label{gpc_solu}
u(\mathbf{x},t;\o) = \sum_{|\i|=0}^p\hat{u_\i}(\mathbf{x},t)\Phi_\i(\bm{\x}(\o)).
\end{equation}
Substituting equation \eqref{gpc_solu} into the governing system \eqref{sto_eq}, we obtain the following system
\begin{equation}\label{sto_eq_gpc}
\sum_{|\i|=0}^p\frac{\partial\hat{u_\i}}{\partial t}\Phi_\i=\mathcal{L}(\mathbf{x},t,\o;\sum_{|\i|=0}^p \hat{u}_\i\Phi_\i),
\end{equation}
By applying Galerkin projection of \eqref{sto_eq_gpc} onto each element of the orthonormal polynomial basis $\{\Phi_\i\}_{|\i|=0}^p$, we derive
\begin{equation}\label{sto_eq_galerkin}
\frac{\partial\hat{u_\i}}{\partial t}=\langle\mathcal{L}(\mathbf{x},t,\o;\sum_{|\i|=0}^p \hat{u}_\i\Phi_\i),\Phi_\j\rangle ,\quad |\j| = 0,1,\cdots,p.
\end{equation}
This is a set of $n$ coupled deterministic equations the random modes $\hat{u}_\i(\mathbf{x},t),|\i| = 0,1,\cdots,p.$ Techniques for deterministic equations can be implemented to solve this system of equations

For smooth problems, the gPC expansion is efficient due to its exponential convergence. However, for non-smooth problems such as nonlinear problems involving discontinuities in random space, gPC expansions extremely slow. gPC expansions can fail to capture the statistical properties of the solution after a short time \cite{WanK_JCP05}.

\subsection{Multi-element gPC representation}
An alternative to global gPC representation for efficiently resolving problems with discontinuities in random space is gPC based on a localization of the random space. These localization methods include, among others, multi-element generalized polynomial chaos(ME-gPC) \cite{WanK_JCP05}, multi-element stochastic collocation \cite{Foo2008,Foo2010,Jakeman2013}, adaptive hierarchical sparse grid collocation \cite{Ma2009} and piecewise polynomial multi-wavelets expansion \cite{MaitreKNG_JCP04,MaitreNGK_JCP04,Pettersson2014}. 

In this paper, we adopt the ME-gPC approach to deal with discontinuities in random space. We briefly introduce the decomposition of the uniform random space (see \cite{WanK_JCP05} for more details).
Let $\bm{\xi} = (\xi_1(\o),\xi_2(\o),\cdots,\xi_d(\o)):\O\longmapsto\mathbb{R}^d$ be a $d$-dimensional random vector defined on the probability space $(\O,\A,\P)$, where $\xi_i,i=1\cdots,d$ are identical independent distributed(i.i.d) uniform random variables defined as $\xi_i:\O\longmapsto[-1,1]$ with probability density function(p.d.f) $f_i=\frac{1}{2}$.
Let $B=[-1,1]^d\subset\mathbb{R}^d$ be decomposed in $N$ non-overlapping rectangular elements as following:
\begin{eqnarray}\label{decom}
&&B_k = [a_1^k,b_1^k)\times[a_2^k,b_2^k)\times\cdots\times[a_d^k,b_d^k],\nonumber\\
&&B = \bigcup_{k=1}^N B_k,\\
&&B_i\cap B_j = \emptyset \quad \text{if } i\neq j,\nonumber
\end{eqnarray}
where $i,j,k = 1,2,\cdots,N$. Let $\chi_k,k=1,2,\cdots,N$ be the indicator random variables on each of the elements defined by
\[
\chi_k = \left\{\begin{array}{ll}1&\text{if }\bm{\xi}\in B_k,\\
0&\text{otherwise.}
\end{array}\right.
\]
$\bigcup_{k=1}^N\chi_k^{-1}(1)$ gives a decomposition of the event space $\O$. For each random element, the local random vector is defined by
\[
\bm{\zeta}^k = (\z_{1}^k,\z_{2}^k,\cdots,\z_d^{k}):\chi_k^{-1}(1)\longmapsto B_k
\]
subject to the conditional p.d.f
\[
f_{\bm{\z}^k} = \frac{1}{2^d\Pb(\chi_k=1)}, \quad k=1,2,\cdots,N,
\]
where $\Pb(\chi_k=1) = \prod_{i=1}^d\frac{b_i-a_i}{2}$.
After that, we transfer each $\bm{\z}^k$ to a new random vector defined on $[-1,1]^d$ by a map $g_k$,
\[
g_k(\bm{\z}^k):\z^k_i = \frac{b_i^k-a_i^k}{2}\z_i^k+\frac{b_i^k+a_i^k}{2},\quad i=1,2,\cdots,d.
\]
Thus,
\[
\bm{\xi}^k = g_k(\bm{\z}^k) = (\z_1^k,\z_2^k,\cdots,\z_d^k):\chi_k^{-1}(1)\longmapsto[-1,1]^d
\]
is the new random vector with constant p.d.f $f^k = \frac{1}{2^d}$.
With this decomposition of the random space of $\bm{\xi}$, we can solve a system of differential equations with random input $\bm{\xi}$ by combining the local approximation via $\bm{\z}^k$ subject to a conditional p.d.f.
In practice if the system solution $u(\bm{\xi})$ is locally approximated by $\hat{u}_k(\bm{\z}^k),k=1,2,\cdots,N$, then the $m$th moment of $u(\bm{\xi})$ on the entire random space can be obtained by
\begin{equation}\label{m_moment}
\begin{split}
\mu_m(u(\bm{\xi})) &= \int_{B}u^m(\bm{\xi})\frac{1}{2^d}d\bm{\xi}\\
&\approx\sum_{k=1}^N\Pb(\chi_k=1)\int_{B_k}\hat{u}_k^m(\bm{\z}^k)f_{\bm{\z}^k}d\bm{\z}^k\\
&=\sum_{k=1}^N\Pb(\chi_k=1)\int_{[-1,1]^d}\hat{u}_k^m(g_k^{-1}(\bm{\xi}^k))\frac{1}{2^d}d\bm{\xi}^k
\end{split}
\end{equation}

%%%%%%%%%%%%%%%End of Setup%%%%%%%%%%%%%%%%%%%

\section{Model reduction and mesh refinement}
This section contains a brief introduction to the main idea behind model reduction and how this can be used to construct a mesh refinement algorithm.

\subsection{Model reduction} \label{sec:mo_reduc}
The Galerkin projection of the stochastic system \eqref{sto_eq} onto the random space transforms it into a deterministic system of coupled equations \eqref{sto_eq_galerkin}. This deterministic system consists, in general, of partial differential equations (PDEs). After spatial discretization the PDEs are replaced by a system of ordinary differential equations(ODEs). This is our starting point for a reduced model.

%\begin{equation}\label{sto_eq_galerkin_t}
%\frac{d\hat{u_\i}}{dt}=<\mathcal{L}(t,\o;\sum_{|\i|=0}^p \hat{u}_\i\Phi_\i),\Phi_\j> ,\quad |\j| = 0,1,\cdots,p.
%\end{equation}
%
We split the set of the random modes into two sets, $F$ of resolved modes and $G$ of unresolved ones. Note that this is an internal splitting of the modes of the system. In what follows we always evolve the total set of modes $F \cup G.$ The main idea behind model reduction is to construct a modified system for the evolution of the modes in $F$ using the modes in $G$ to effect the necessary transfer of activity between $F$ and $G.$ 

One can construct a reduced model for the modes in $F$, for example, by using the Mori-Zwanzig formalism \cite{ChorinHK00}. Let $U = (\{\hat{u}_\i\}), \i\in F\cap G$ be the vector of all random modes. The system of ODEs for their evolution can be written as
\begin{equation}\label{ode_R}
\frac{d U(t)}{dt} = R(t, U(t)),
\end{equation}
where $R(t,U(t))$ is the appropriate right hand side (RHS) after all the necessary discretizations. Let $\hat{U}$ denote the vector of resolved modes and $\t{U}$ denote the vector of unresolved modes. Similarly, the RHS denoted by $R(t,U) = (\hat{R}(t,U),\t{R}(t,U))$. Model reduction constructs a modified system for the evolution of the modes in $\hat{U}$ which should follow accurately these modes without having to solve for the full system. Inevitably, the modified system contains an approximation of the dynamics of the unresolved modes $\tilde{U}.$ However, a good reduced model will capture accurately the transfer of activity between the modes in $\hat{U}$ and $\tilde{U}.$ It is this property of a good reduced model that we will exploit in constructing our mesh refinement algorithm.

\subsection{The mesh refinement algorithm}\label{sec:method}
Consider a system of equations with dependence on some random parameters. We decompose the random space in elements as in \eqref{decom}. For each element we consider a system of equations as in \eqref{ode_R} resulting from a gPC expansion of the solution within this element. The associated $L_2$ norm for the modes in $F$ only is $\hat{E} = \sum_{\k\in F}|\hat{u}_{\k}|^2.$ We construct a reduced model for the modes in $F$ which is given by a {\it new} system of equations

\begin{equation}\label{ode_R_reduced}
\frac{d \hat{U'}(t)}{dt} = \hat{R}'(t, U'(t)).
\end{equation}
Note that the RHS will be different from the RHS of the equations for $\hat{U}$ in \eqref{ode_R}. The associated $L_2$ norm for the modes in $F$ is $\hat{E}' = \sum_{\k\in F}|\hat{u}'_{\k}|^2.$ We can define such an $L_2$ norm for each element in the random space. We monitor $|\frac{d\hat{E}'}{dt}|$ that is, the absolute value of the rate of change of the quantity $\hat{E}'$ in each element. When this exceeds a prescribed tolerance we stop and refine.  

%With all these preparations, we can implement the mesh refinement algorithm to the random space.
%Consider a stochastic system with conservation of energy and can be numerical solved via Galerkin projection, and let $\hat{E}$ denote the energy of the truncated Galerkin system, then $\hat{E}$ is a time invariant quantity. Immediately, we have the rate of change of $\hat{E}$ $(\frac{d\hat{E}}{dt})$ equals $0$. For the same Galerkin system if we can construct a robust reduced system and
%let $\hat{E}'$  be the energy derived from the reduced model. If the reduced system is a prefect reduction model of the original system then $\frac{d\hat{E}'}{dt} = 0$. Most of the case, for a robust reduced model $\frac{d\hat{E}'}{dt} < 0$. It means that for a reduced system with less modes, energy starts draining out from the resolved modes, then more modes are required or more elements are desired. Meanwhile, $\hat{E} = \sum_{\k\in F}|\hat{u}'_{\k}|^2$ has the direct connect with the variance of $\hat{u}$. As a result $|\frac{d\hat{E}'}{dt}|$ is a measuring whether the reduced system capture the properties of the original system. It follows that $|\frac{d\hat{E}'}{dt}|$ can be chosen as the quantity to determine when the element need the refinement. 

If the random space has $d$ dimensions with $d \ge 2$ then we need to decide not only when and where it is time to refine but also in which direction. Since the modes of the polynomial basis with highest degree contribute most to the transfer of activity from $F$ to $G,$ we define $s_i = |\frac{d|\hat{u}'_{p_r\mathbf{e}_i}|^2}{dt}|$ with $i=1,\ldots,d$ to denote the contribution of the $i$th random dimension to $|\frac{d\hat{E}'}{dt}|.$ Here, $p_r\mathbf{e}_i$ is the index vector with the highest degree $p_r$ in $i$th-dimension and degree zero in the rest of the dimensions.

\vspace{1cm}

%\begin{algorithm}
\begin{enumerate}
\item[]{\bf Mesh refinement algorithm}
\item[Step 1] Choose values $TOL_1>0$ and $TOL_2>0$ for the tolerances.
%\item[Step 2] Construct a stochastic ODEs system by gPC and the reduced model of this ODEs system.
\item[Step 2] Mesh refinement:\\
%\begin{algorithmic}
For time step $t \leftarrow 1, \cdots, N$ \\
\begin{itemize}
\item[] Loop over all elements:
\begin{itemize}
\item[] On the $k$-th element $B_k$, update the modes for $B_k$\\
%\begin{itemize}
 If{$(|\frac{d\hat{E}'}{dt}|)\Pr(B_k)\geq TOL_1$,} loop over all dimensions:
\begin{itemize}
\item[] If {$s_i\geq TOL_2\cdot \max_{j=1,\cdots, d} s_j$,}
\begin{itemize}
\item[] split the element $B_k$ in two equal parts along the $i$th dimension and generate local random variables $\xi_{i,1}$ and $\xi_{i,2}$
\end{itemize}
\item[] End if
\end{itemize}
\item[] End if
\item[] Update the information of the new elements
\end{itemize}
\item[] End loop
\end{itemize}
\end{enumerate}

\vspace{1cm}

We should make a few remarks about the algorithm. First, we do not need to compute the rate of change by numerical differentiation which is inaccurate. Instead, by using the RHS of the equations \eqref{ode_R_reduced} we obtain an expression for the rate of change which does not involve temporal derivatives (see Appendix A for the relevant expressions for the Kraichnan-Orszag system).

Secondly, for each element we monitor $(|\frac{d\hat{E}'}{dt}|)\Pr(B_k)$ and not just $|\frac{d\hat{E}'}{dt}|.$ This is because each element should be weighted appropriately so that there is no excessive refinement for elements whose contribution is negligible (see \eqref{m_moment}). 

Thirdly, there are two ways to compute $|\frac{d\hat{E}'}{dt}|.$ One way is to evolve, for each element $B_k$, both the full and reduced systems of equations \eqref{ode_R} and \eqref{ode_R_reduced}. One then uses the values of the modes in $F$ from the reduced system to compute the expressions involved in $|\frac{d\hat{E}'}{dt}|.$ The second way is to evolve {\it only} the full system \eqref{ode_R} and then use the values of the modes in $F$ to compute the expressions  in $|\frac{d\hat{E}'}{dt}|.$

%%%%%%%%%%%%%%%%%%%End of Method%%%%%%%%%%%%%%%

\section{Numerical Examples} \label{sec:examples} We present results of our mesh refinement algorithm for a simple linear ODE with a random parameter and for the Kraichnan-Orszag three-mode system with random initial conditions.

\subsection{One-dimensional ODE}
We begin by considering the simple ODE
\begin{equation}\label{ex:ODE}
\frac{du}{dt} = -\kappa(\omega)u,\quad u(0;\omega) = u_0,
\end{equation}
where $\kappa(\omega)\sim U(-1,1)$. The exact solution of this equation is
\begin{equation}\label{ex:ODE_sol}
u(t,\omega) = u_0e^{-\kappa(\omega)t}.
\end{equation}
The statistical mean of the solution is $$\mu(u(t;\o)) = \{\begin{array}{ll} \frac{u_0}{2t}(e^{t}-e^{-t}),& t>0\\u_0,&t=0\end{array},$$ and the variance is $$\sigma^2(u(t;\o))=\{\begin{array}{ll}\frac{u_0^2}{4t}(e^{2t}-e^{-2t})-\frac{u_0^2}{4t^2}(e^{2t}+e^{-2t}-2),&t>0\\0,& t=0\end{array}$$
Assume $$ \tilde{u}=  \sum_{i=0}^N\tilde{u}_i\Phi_i(\xi(\o)),\quad \kappa(\o) = \sum_{i=0}^N \kappa_i \Phi(\xi(\o))$$ where $\{\Phi_i\}$ are the orthonormal Legendre polynomial chaos bases and $\xi(\o)\sim U(-1,1)$. The coefficients of the gPC expansion satisfy the ODE system
\begin{equation}\label{ex:ODE_galer}
\frac{\tilde{u}_k}{dt} = -\sum_{i=0}^N\sum_{j=0}^N u_0\kappa_i \tilde{u}_je_{ijk}, \quad k=0,\cdots,N,
\end{equation}
where $$e_{ijk} = \int_{\Omega}\Phi_i(z)\Phi_j(z)\Phi_k(z)\rho(z)dz.$$ The solution $\tilde{u}$ is an approximation of $u$.
To implement the adaptive mesh refinement, first we need to construct a reduced model. We have chosen the $t$-model which was originally derived through the Mori-Zwanzig formalism and has been thoroughly studied \cite{ChorinHK02,HaldS07}. For the system \eqref{ex:ODE_galer}, the $t$-model reads

\[
\begin{split}
\frac{d\tilde{u}_k}{dt} =& - u_0\sum_{i\in F\cup G}\sum_{j\in F}\kappa_i \tilde{u}_je_{ijk}\\
&+tu_0^2\sum_{i\in F\cup G}\sum_{j\in G}\sum_{s\in F\cup G}\sum_{t\in F}\kappa_i\kappa_s e_{ijk}e_{stj}\tilde{u}_{t}, \quad k \in F
\end{split}
\]

where $F = \{0,1,\cdots,p_r\}$ is the set of indices of the resolved modes, and $G = \{p_r+1,\cdots,p_f\}$ is the set of indices of the unresolved modes. We study the evolution of the mean and the variance of the solution to the system. There is no discontinuity involved in this problem. However as time increases the lower ordered gPC solution starts to deviate from the exact solution. One way to keep the error under control is to increase the order of the gPC expansion. Another way is to divide the random space into smaller elements so that a lower order expansion is adequate. For our adaptive mesh refinement algorithm we have chosen a low order gPC expansion for each element. The price one pays is the need to solve more small systems instead of a large one. Figure \ref{fig:ode} shows the curves of the mean(left) and variance(right) via various methods. The results from the refinement algorithm almost reproduce the exact solution while the global gPC solution starts departing from the exact solution as time evolves. We choose the mean and variance of the exact solution as references to study the relative error of each algorithm. We define
$$\textrm{Error of mean} = |\frac{\mu(u(t;\o))-\mu(\tilde{u}(t;\o))}{\mu(u(t;\o))}|,$$
$$\quad\textrm{Error of variance} = |\frac{\sigma^2(u(t;\o))-\sigma^2(\tilde{u}(t;\o))}{\sigma^2(u(t;\o))}|.$$
In Figure \ref{fig:ode_er}, the evolutions of the error of the gPC solution with order 3 and ME-gPC are shown for different values of the accuracy control parameters and different orders of the reduced models. The maximum relative errors for the mean and the variance resulting from the use of reduced models of different orders are presented in Table \ref{tab:ode}. For reasons of comparison we include the relative error from the gPC solution of order 3. As we can see higher ordered reduced models require less elements and at the same time obtain better accuracy. The adaptive meshes at $t=10$ with different accuracy control values $TOL_1=10^{-1}$ and $TOL_1 = 10^{-2}$ are demonstrated in Figure \ref{fig:ode_mesh}. The elements on the left end are smaller than those on the right end. This is consistent with the fact that the rate of change of $u(t;\o)$ is larger on the left end of the random space. The evolution of the error, of the number of elements and of the error for the variance are shown in Figure \ref{fig:ode_error_mesh}.
\begin{figure}[htbp]
   \centerline{
   \psfig{file=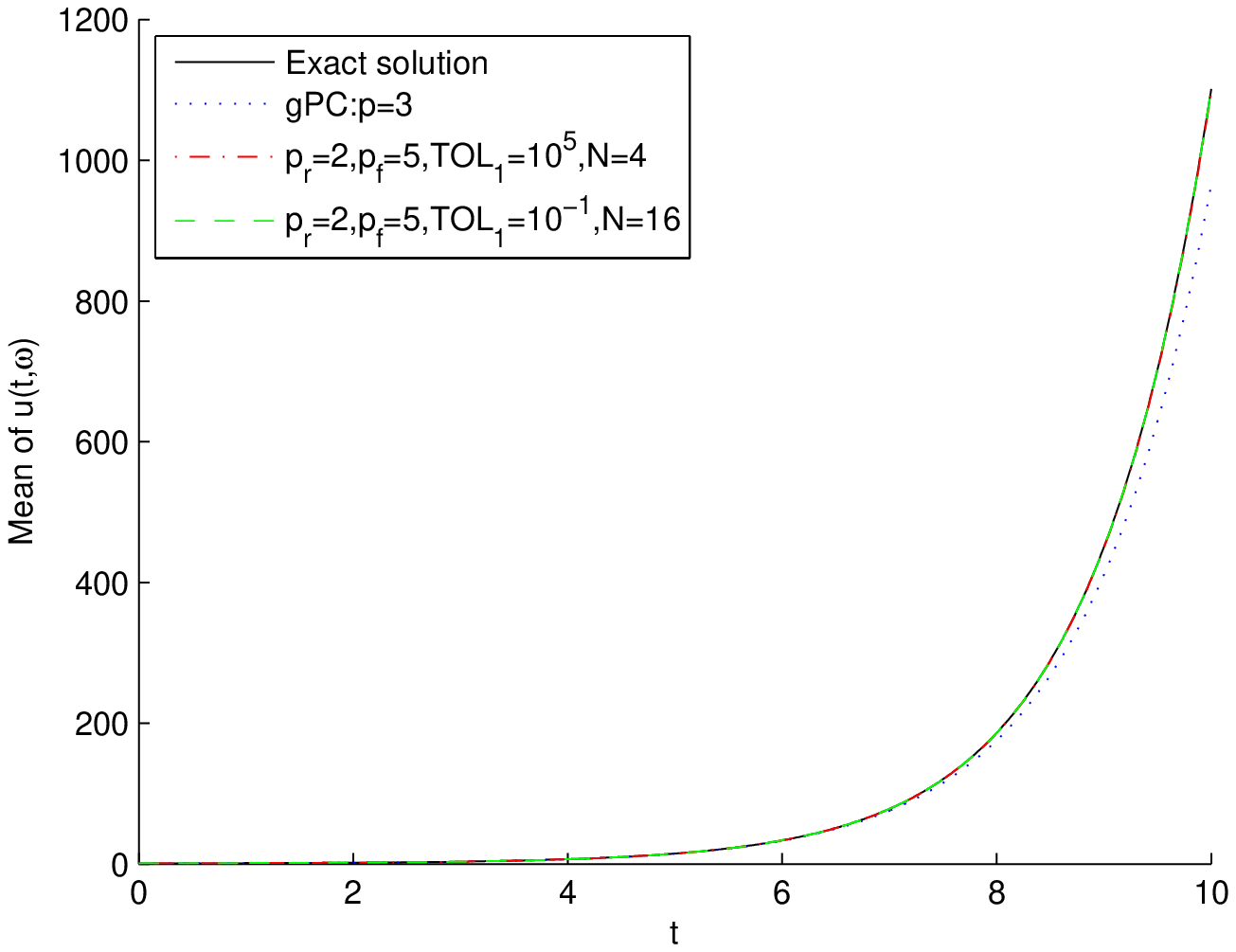,width=7cm}
   \psfig{file=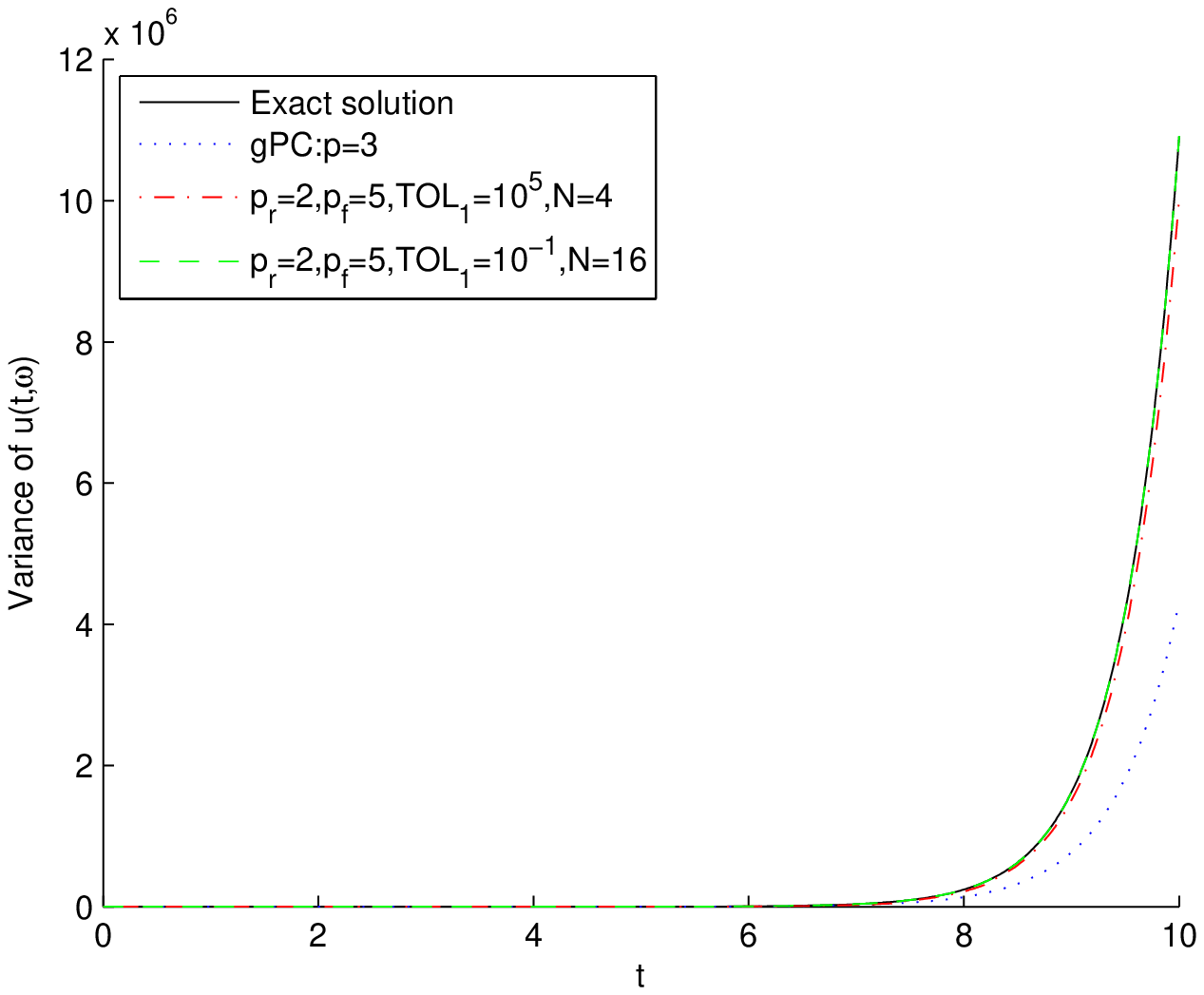,width=7cm}
  }
\caption{Evolution of mean of $u(t;\o)$(left) and evolution of variance of $u(t;\o)$(right) for the simple ODE.
  }
\label{fig:ode}
\end{figure}

\begin{figure}[htbp]
   \centerline{
   \psfig{file=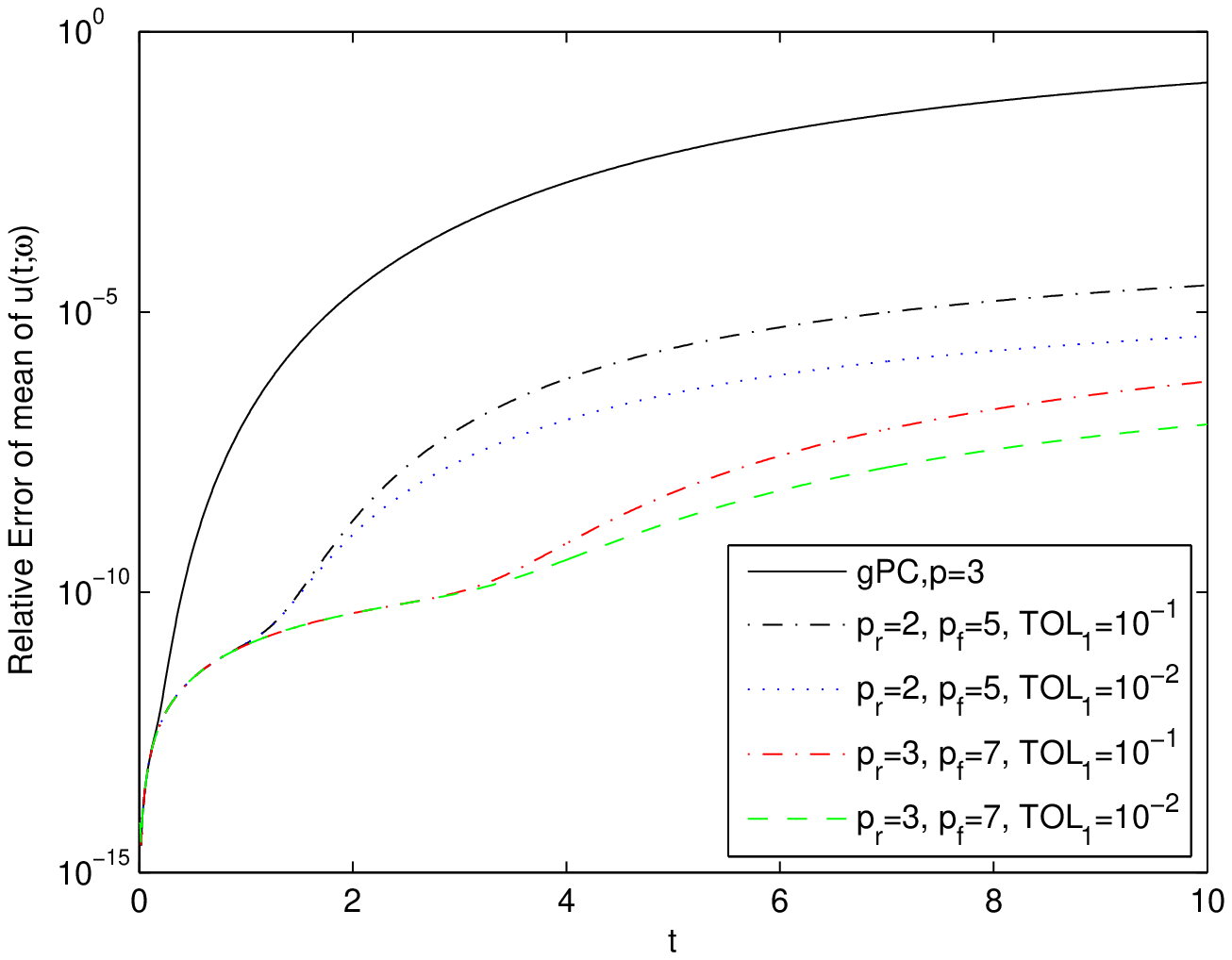,width=7cm}
   \psfig{file=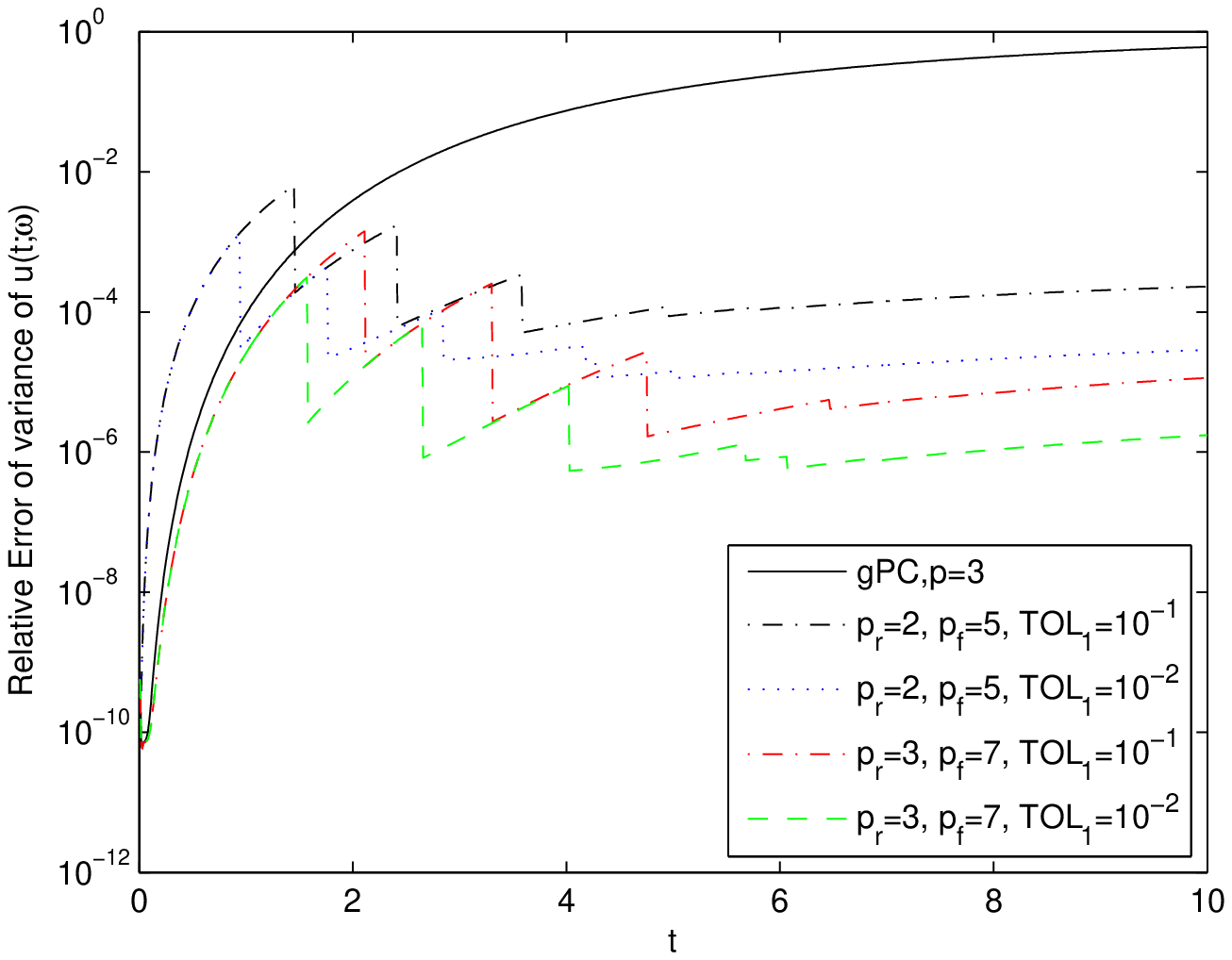,width=7cm}
  }
\caption{Evolution of relative error of mean of $u(t;\o)$(left) and evolution of relative error of variance of $u(t;\o)$(right) for the simple ODE.
  }
\label{fig:ode_er}
\end{figure}

\begin{table} [htbp]
\begin{center}
\begin{tabular}{llccc}
\toprule
        & N & Error of $\mu(u)$ & Error of $var(u)$\\
\midrule
gPC, $p=5$& $1$ & $3.8e-3$ & $1.1e-1$\\
\midrule
$TOL_1=10^{-1}$\\
$p_r=2,p_f=5$ & $16$ & $3.0e-5$ & $6.4e-3$\\
$p_r=3,p_f=7$ & $9$ & $5.7e-7$ & $1.4e-3$\\
\midrule
$TOL_1=10^{-2}$\\
$p_r=2,p_f=5$ & $23$ & $3.7e-6$ & $1.3e-3$ \\
$p_r=3,p_f=7$ & $12$ & $9.8e-8$ & $3.1e-4$ \\
\bottomrule
\end{tabular}
\caption{Maximum of relative error of mean and variance of solution to the simple ODE when $t\in(0,10]$}
\label{tab:ode}
\end{center}
\end{table}

\begin{figure}[htbp]
   \centerline{
   \psfig{file=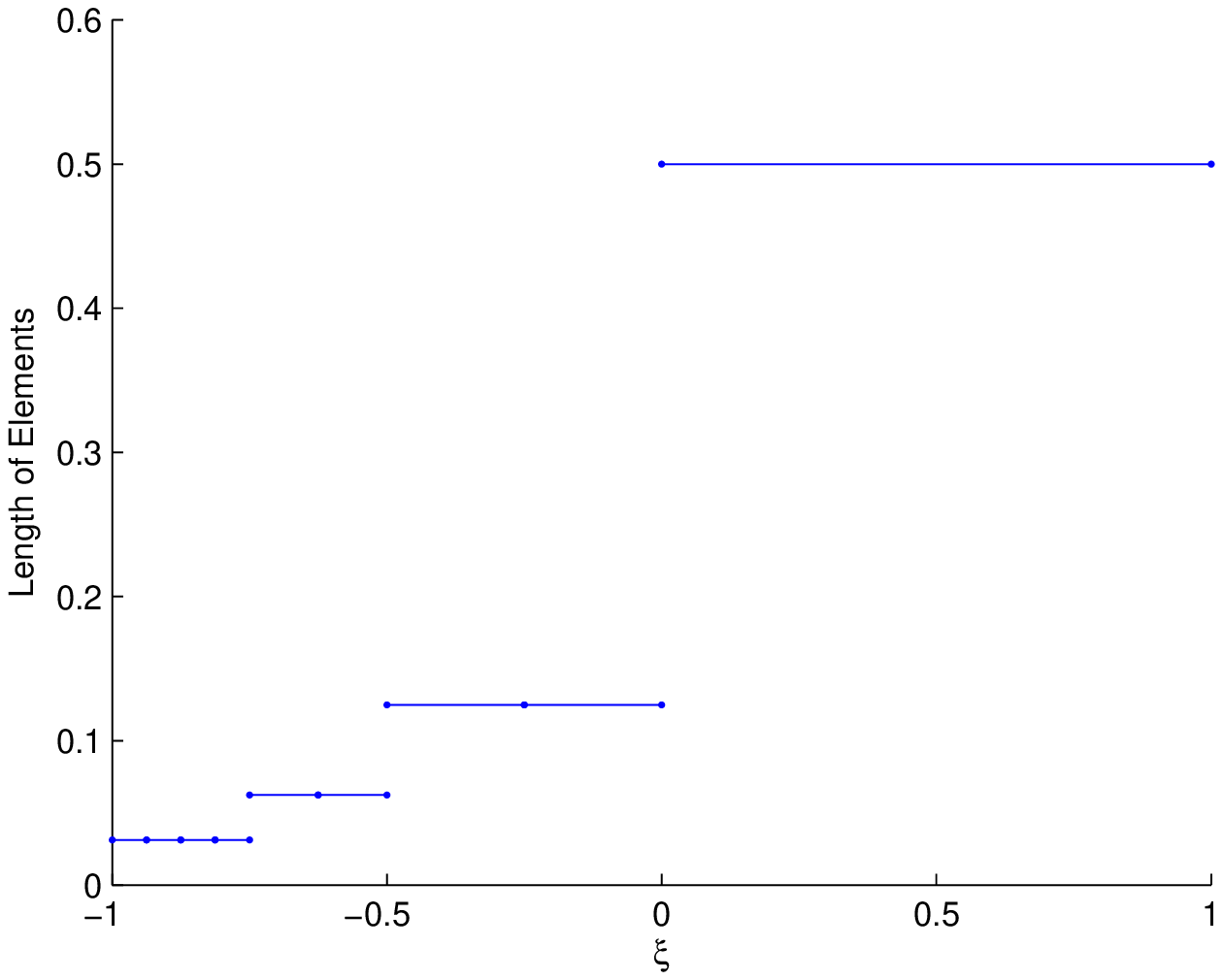,width=7cm}
   \psfig{file=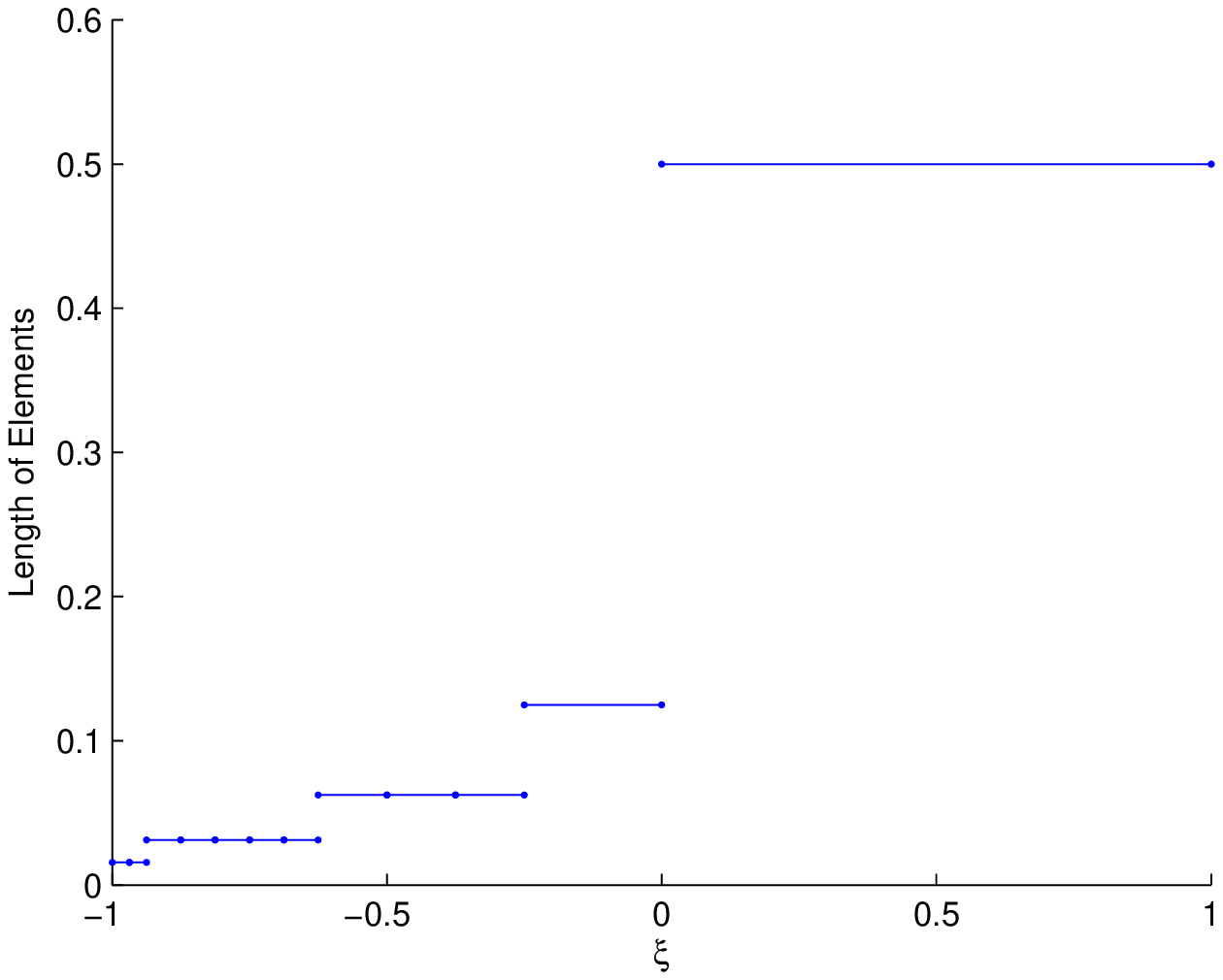,width=7cm}
  }
\caption{Adaptive meshes for the simple ODE when $p_r=3,p_f=7$, $TOL_1=10^{-1}$(left) and $TOL_1=10^{-2}$(right).
  }
\label{fig:ode_mesh}
\end{figure}

\begin{figure}[htbp]
   \centerline{
   \psfig{file=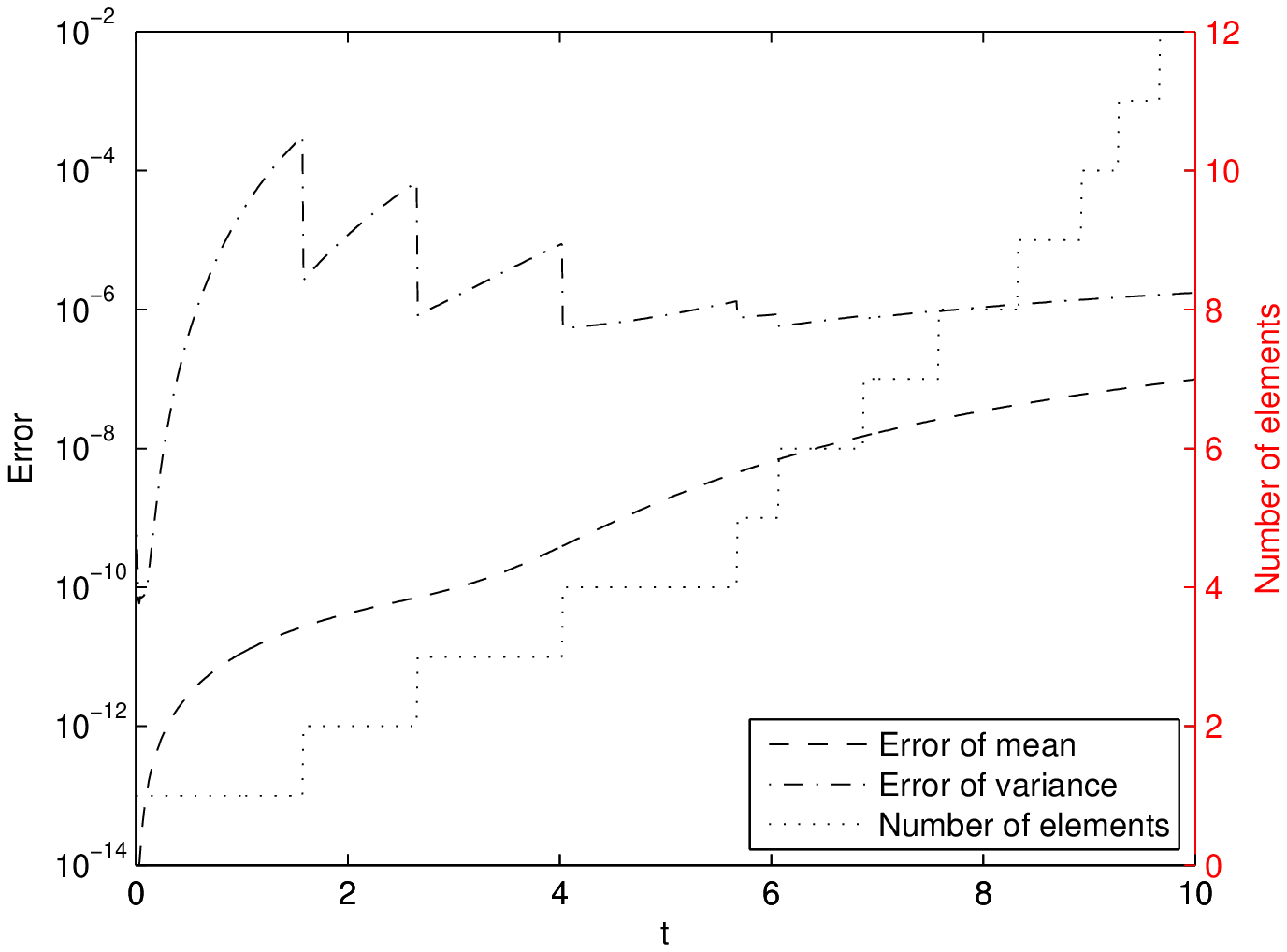,width=11cm}
  }
\caption{Evolutions of mean error and variance error compared with adaptive meshes for the simple ODE when $p_r=3,p_f=7$, $TOL_1=10^{-2}$.
  }
\label{fig:ode_error_mesh}
\end{figure}

%%%%%%%%%%%%%%%%%%%%%%%%%%%%%%%%%%%%%%%%%%%%%%%%%%%%%%%%%%%%%%%%%%%%%%%%%%%%%%%%%%%%%%%%%%%%%%%%%%%%%%%%%%%%%%%%%%%%%%%%%%%%%%%%%%%%
\subsection{The Kraichnan-Orszag three-mode system}
In \cite{Orszag1967} it was shown that a Wiener-Hermite expansion does not faithfully represent the dynamics of the system when the random inputs are Gaussian random variables. A mesh refinement algorithm can efficiently quantify the uncertainty of the system when the random inputs are uniform random variables \cite{WanK_JCP05}. For computational convenience, we consider the following system obtained by a linear transformation performed on the original Kraichnan-Orszag three-mode system.
\begin{equation}\label{ex:KO}
\begin{split}
\frac{dy_1}{dt} &= y_1y_3,\\
\frac{dy_2}{dt} &= -y_2y_3,\\
\frac{dy_3}{dt} &= -y_1^2+y_3^2,
\end{split}
\end{equation}
with the initial conditions
\begin{equation}\label{ex:KO_ini}
y_1(0) = y_1(0;\omega),\quad y_2(0) = y_2(0;\omega),\quad y_3(0) = y_3(0;\omega),
\end{equation}
The discontinuity occurs at the planes $y_1 =0$ and $y_2=0.$ Similarly to \cite{WanK_JCP05}, we consider the case with one random input, two random inputs and three random inputs respectively. Both the Galerkin ODEs and the reduced model of the Kraichnan-Orszag three-mode system are derived in \ref{app:gpc}.

\subsubsection{One-dimensional random input}
We choose the initial conditions
\begin{equation}\label{ex:KO_ini_1d}
y_1(0;\o) = 1,\quad y_2(0;\o) = 0.1\xi(\o),\quad y_3(0;\omega)=0,
\end{equation}
where $\xi\sim U[-1,1]$. In this case the discontinuity point $y_2=0$ is in the random input space.

\begin{figure}[htbp]
   \centerline{
   \psfig{file=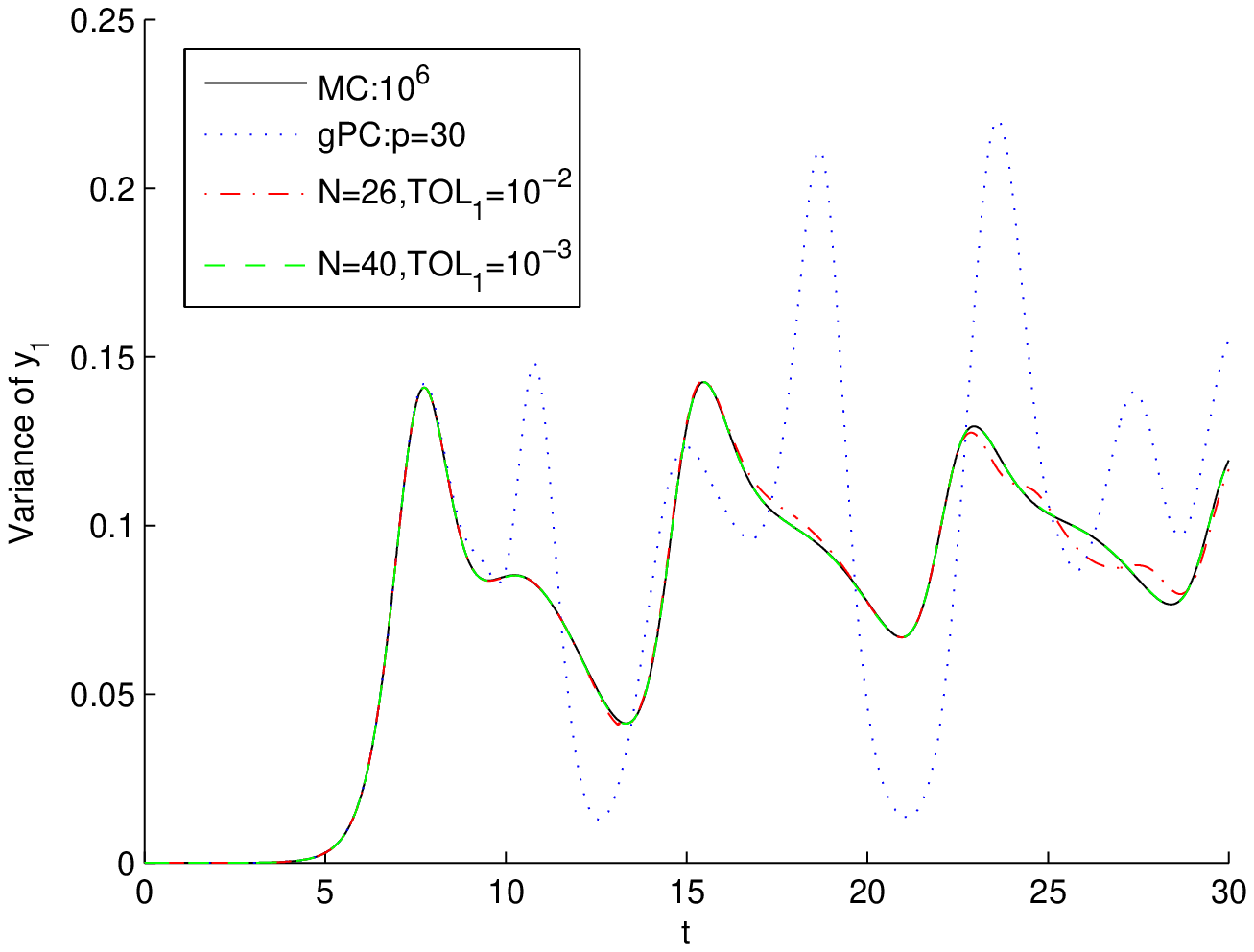,width=7cm}
   \psfig{file=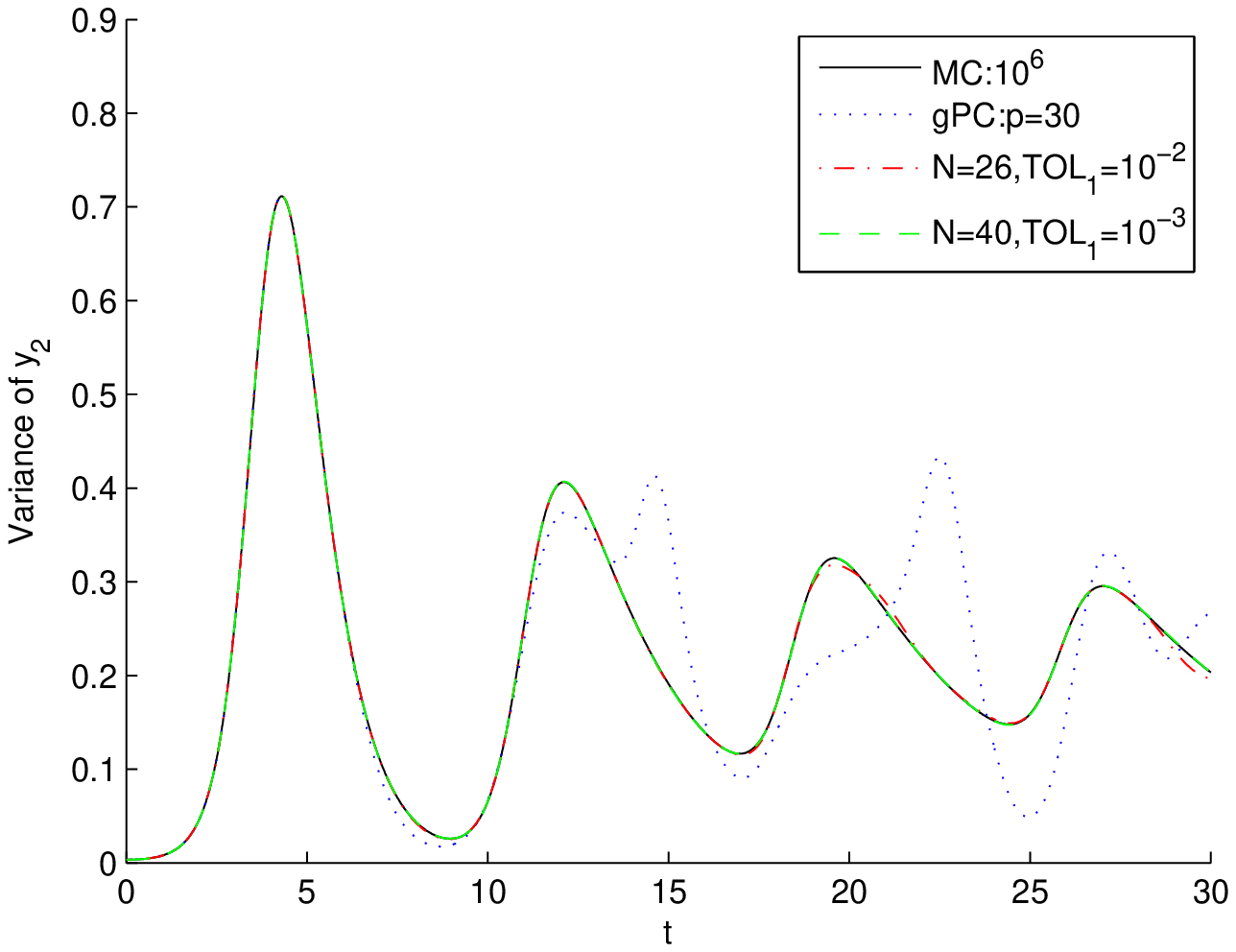,width=7cm}
  }
\caption{Evolution of the variance of $y_1$(left) and evolution of the variance of $y_2$(right) for the Kraichnan-Orszag three-mode system with 1D random inputs.
  }
\label{fig:KO_y1y2_1d}
\end{figure}

 We study the variance of each random output $y_i$, $i=1,2,3$ on the time interval $[0,30]$. Figure \ref{fig:KO_y1y2_1d} presents the variance evolution of $y_1$ and that of $y_2$ estimated by Monte Carlo simulation with $1,000,000$ samples and adaptive mesh refinement ME-gPC with polynomial basis order 3 of reduced model and order 7 of full model under various accuracy control values $TOL_1$. Failure to capture the properties via the global gPC expansion after a short time is also shown in Figure \ref{fig:KO_y1y2_1d}. If we keep the order of the expansion constant and make the tolerance $TOL_1$ stricter, more elements are needed by the mesh refinement algorithm and higher accuracy is achieved.

\begin{table} [htbp]
\begin{center}
\begin{tabular}{llccc}
\toprule
        & N & Error of $var(y_1)$ & Error of $var(y_2)$ & Error of $var(y_3)$\\
\midrule
$TOL_1=10^{-3}$\\
$p_r=3,p_f=7$ & $40$ & $1.3e-3$ & $3.0e-3$ & $1.7e-3$\\
$p_r=4,p_f=9$ & $34$ & $4.2e-3$ & $4.0e-3$ & $2.2e-3$ \\
$p_r=5,p_f=11$ & $32$ & $2.5e-3$ & $4.1e-3$ & $2.1e-3$\\
\midrule
$TOL_1=10^{-6}$\\
$p_r=3,p_f=7$ & $110$ & $7.4e-7$ & $6.1e-6$ & $6.7e-5$\\
$p_r=4,p_f=9$ & $80$ & $3.8e-7$ & $2.8e-7$ & $6.1e-7$\\
$p_r=5,p_f=11$ & $64$ & $1.6e-7$ & $1.0e-6$ & $1.2e-6$\\
\midrule
$TOL_1=10^{-9}$\\
$p_r=3,p_f=7$ & $256$ & $4.0e-9$ & $9.4e-9$ & $2.0e-6$\\
$p_r=4,p_f=9$ & $170$ &  $6.8e-10$ & $4.5e-9$ & $6.2e-9$\\
$p_r=5,p_f=11$ & $128$ &  $1.3e-10$ & $1.0e-9$ & $1.4e-8$\\
\bottomrule
\end{tabular}
\caption{Maximum relative errors for the variance of $y_1,y_2$ and $y_3$ when $t\in[0,30]$ for the Kraichnan-Orszag three-mode system with 1D random input}
\label{tab:KO_1d}
\end{center}
\end{table}

\begin{figure}[ht]
   \centering
   \subfigure[]{%
   \includegraphics[width = 5.5cm]{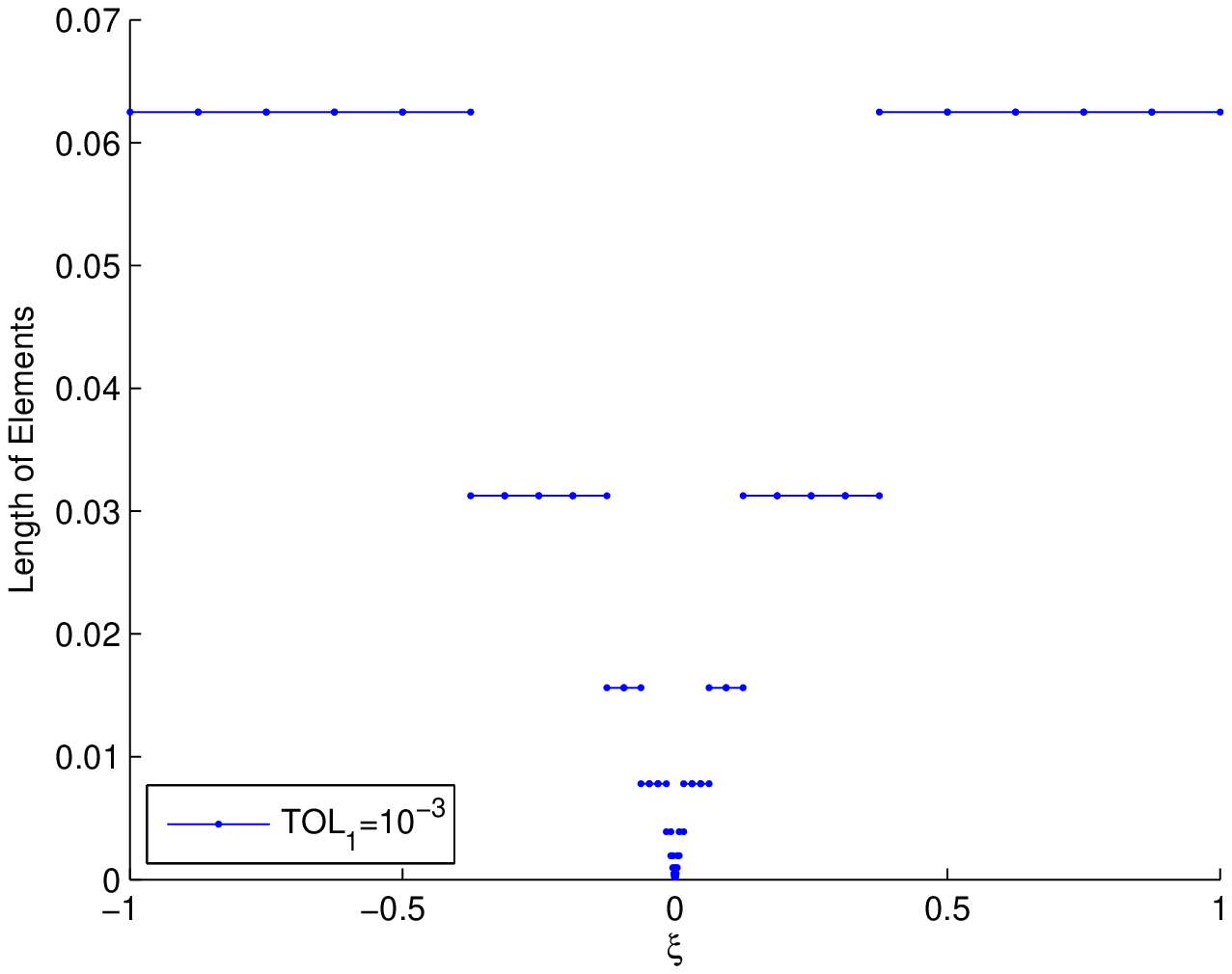}
   \label{KO_mesh_1d3_a}}
   \quad
   \subfigure[]{%
   \includegraphics[width = 5.5cm]{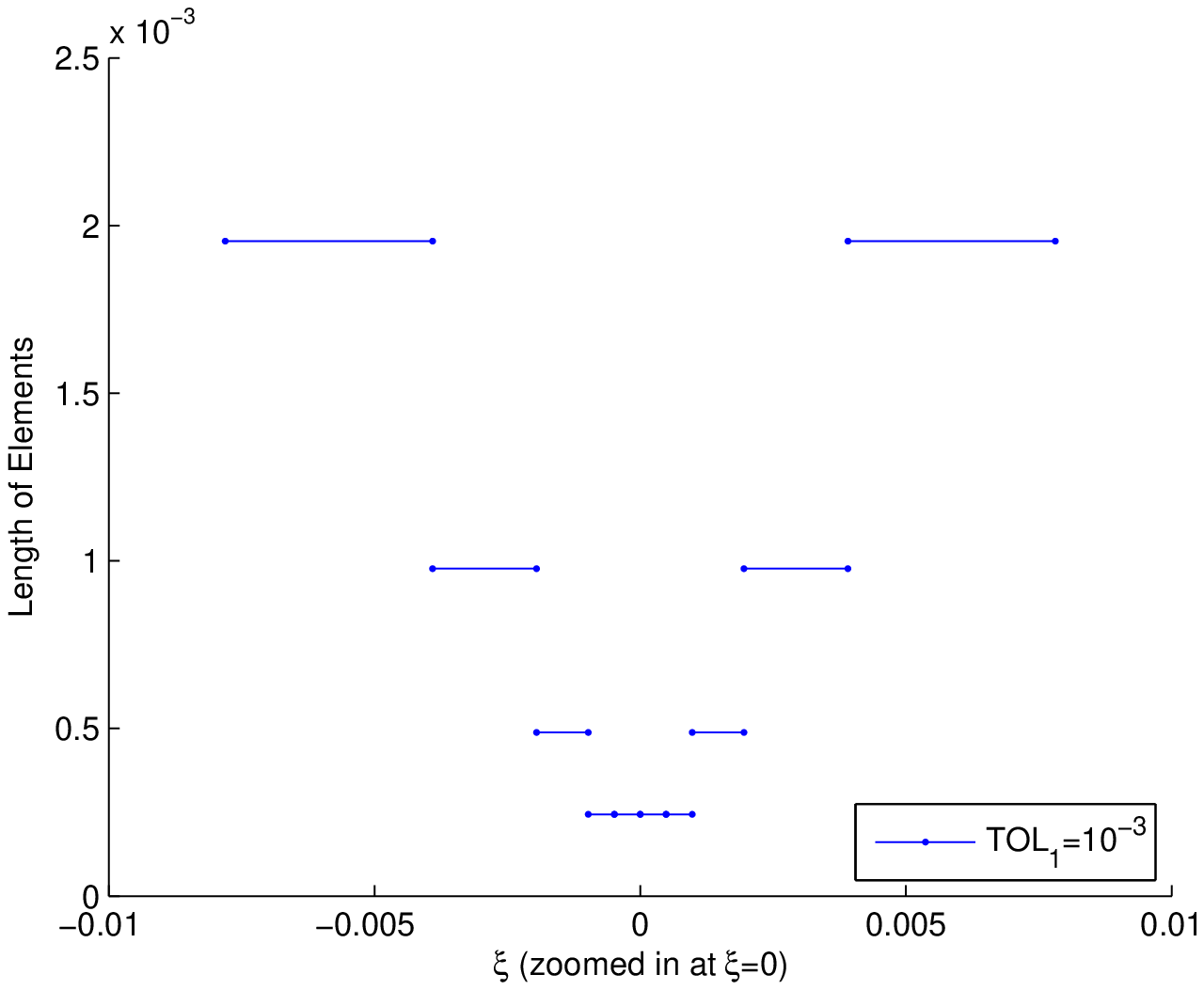}
   \label{KO_mesh_1d3_b}}

   \subfigure[]{%
   \includegraphics[width = 5.5cm]{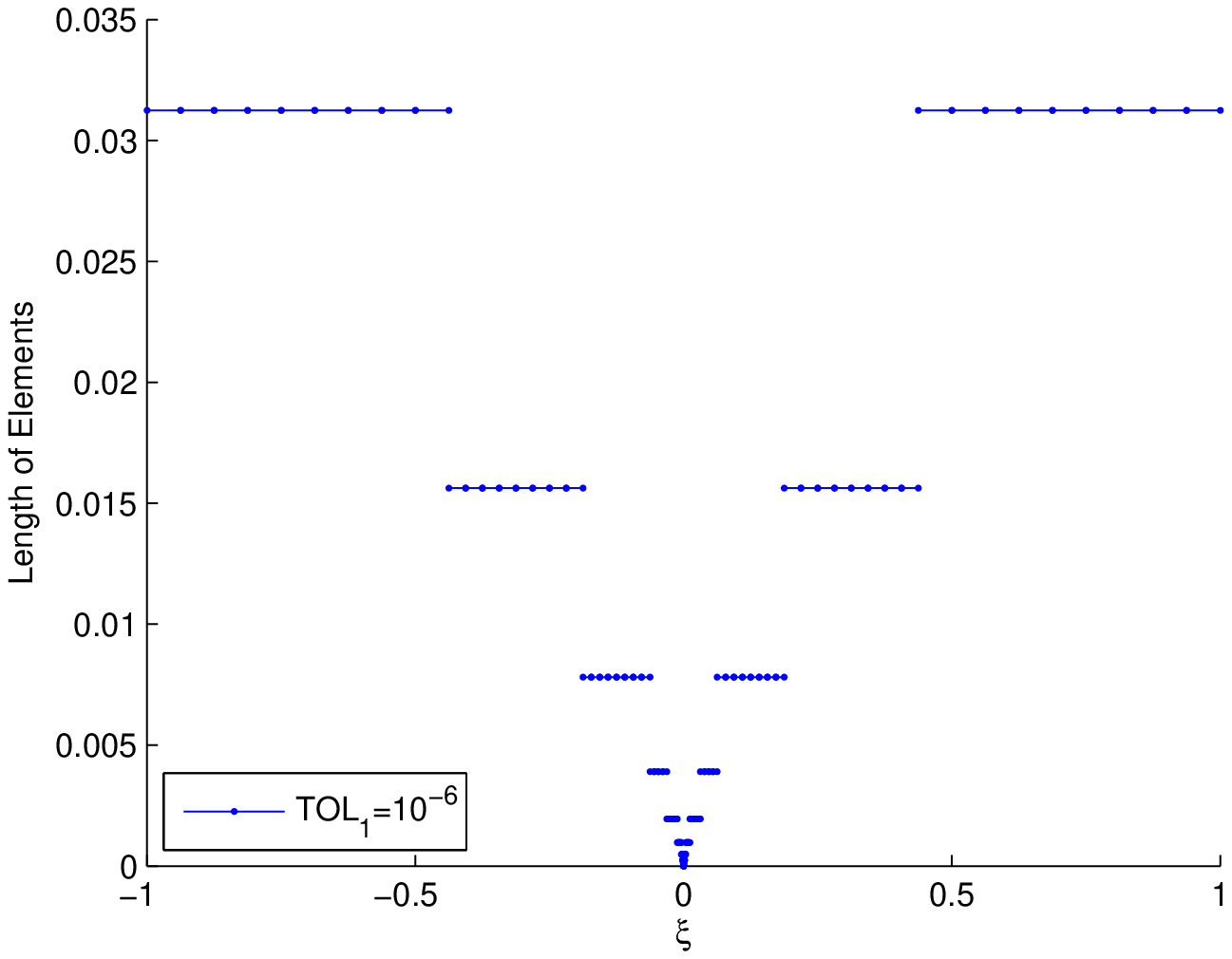}
   \label{KO_mesh_1d6_a}}
   \quad
   \subfigure[]{%
   \includegraphics[width = 5.5cm]{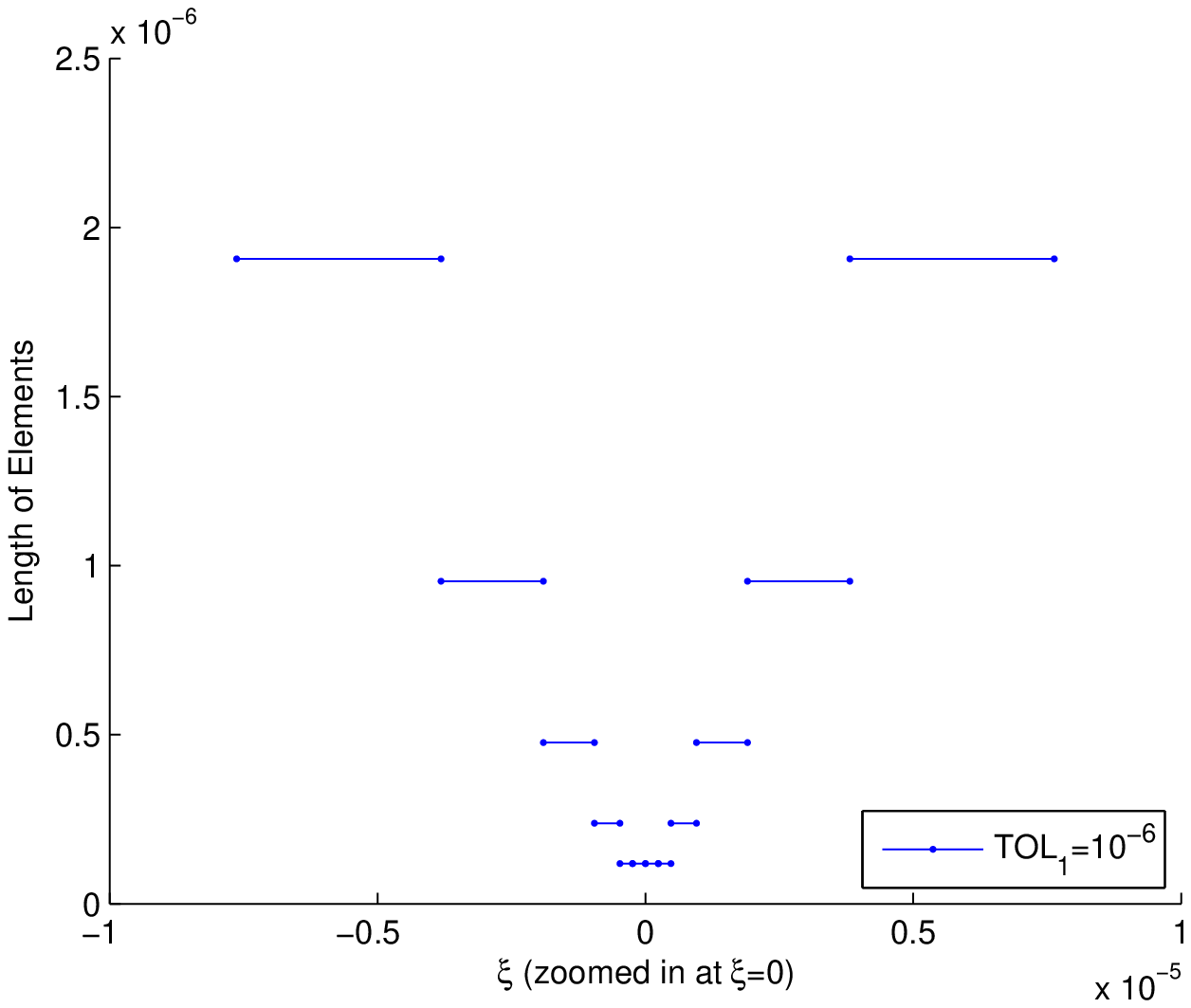}
   \label{KO_mesh_1d6_b}}
\caption{Adaptive meshes for the Kraichnan-Orszag three-mode system with 1D random input when $p_r=3,p_f=7$. (a) $TOL_1=10^{-3},N=40$; (b) zoom-in mesh of (a) near $\xi=0$; (c) $TOL_1=10^{-6},N=110$; (d) zoom-in mesh of (c) near $\xi=0$.
  }
\label{fig:KO_mesh_1d}
\end{figure}

Since ME-gPC achieves higher accuracy than the original gPC on the Kraichnan-Orszag three-mode system, we use the numerical results of reduced model of order $7$ and full model of order $15$ and $TOL_1 = 10^{-12}$ as the reference to exam the errors of different sets of reduced model and full model orders (see Table \ref{tab:KO_1d}). For a fixed value of the tolerance $TOL_1,$ higher order models require fewer elements.

Details of the adaptive meshes from our ME-gPC algorithm around $\xi=0$ are presented in Figure \ref{fig:KO_mesh_1d}. The finest meshes are around the discontinuity of the random space. It demonstrates that our mesh refinement criterion identifies accurately the discontinuity even though the elements are small. Furthermore, when $TOL_1$ is extremely small, the meshes exhibit the pattern that the closer the element is to $\xi=0,$ the smaller the element. Meanwhile the meshes are symmetric with respect to $\xi=0$ as they should be according to the symmetry of the system.

\subsubsection{Two-dimensional random input}
We study the system with initial conditions involving two independent random inputs
\begin{equation}\label{ex:KO_ini_2d}
y_1(0;\o) = 1,\quad y_2(0;\o) = 0.1\xi_1(\o),\quad y_3(0;\omega)=\xi_2(\o),
\end{equation}
where $\xi_1$ and $\xi_2$ are independent uniform random variables on $[-1,1]$.
In Figure \ref{fig:KO_mesh_2d}, we plot the evolution of the variance of each random output $y_i,i=1,2,3$ subject to a 2D random input and show the mesh of the random space at time $t=10$ generated by order 3 reduced model and order 7 full model with $TOL_1=10^{-3}$ and $TOL_2=0.1$. The smallest elements are around the discontinuity $y_2=0$ and the results are more sensitive to $\xi_1$ because of the discontinuity introduced by $\xi_1$. The results for $p_r=7$ and $p_f=9$ with accuracy control $TOL_1 = 10^{-7}$ and $TOL_2=0.1$ are selected to be the references to derive the relative errors to low ordered models. Table \ref{tab:KO_2d} presents the relative errors of the variance of $y_1,y_2$ and $y_3$ for different models and different levels of accuracy. We observe similar trends as in the 1D case, namely that more accurate models require fewer elements if the accuracy tolerance is kept fixed. Also, that stricter accuracy control requires more elements if the order of the models is fixed. To gain the same level of relative errors, the number of the elements increases faster in the 2D case than the 1D case.
\begin{figure}[ht]
   \centering
   \subfigure[]{%
   \includegraphics[width = 5.5cm]{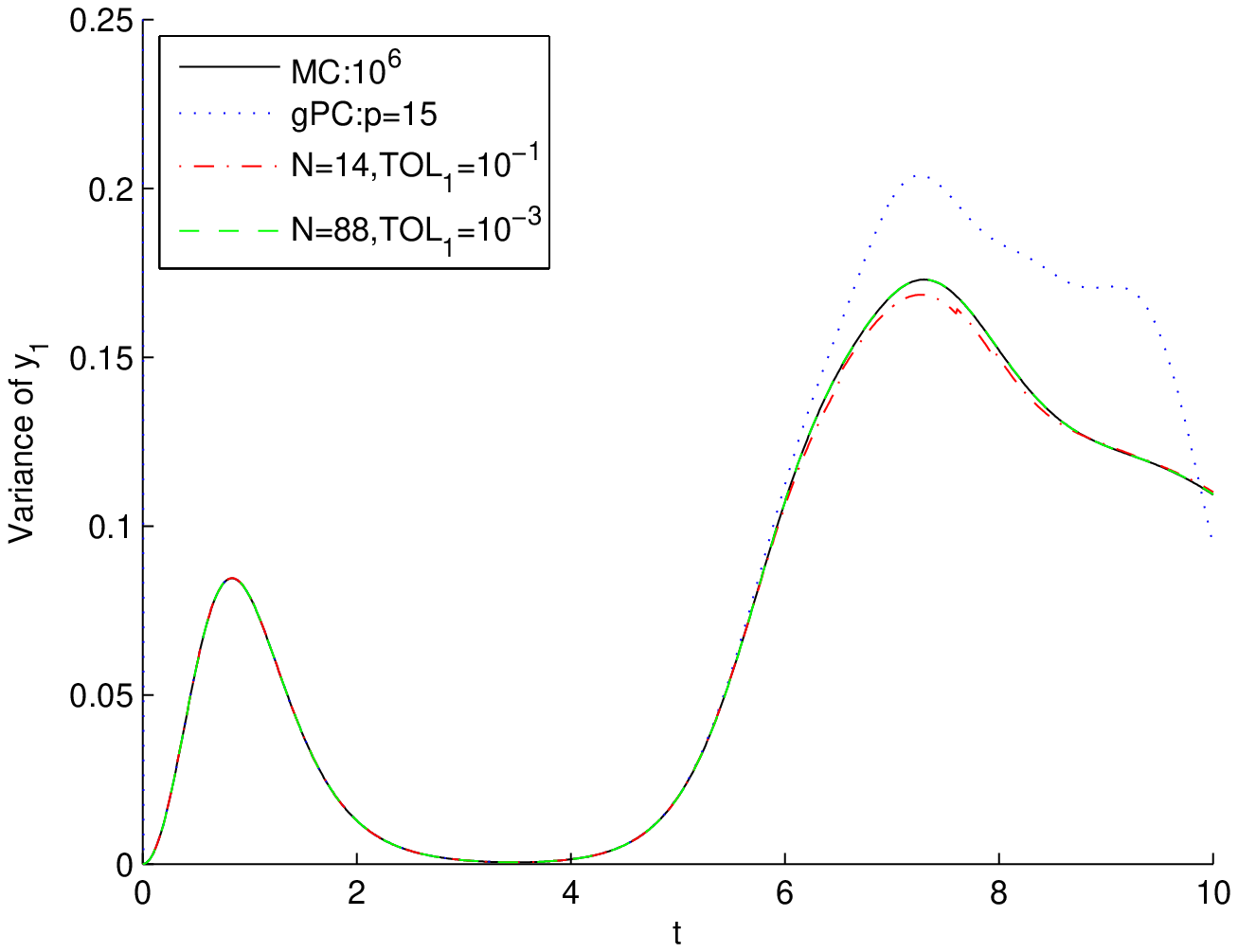}
   \label{y1_var_2d}}
   \subfigure[]{%
   \includegraphics[width = 5.5cm]{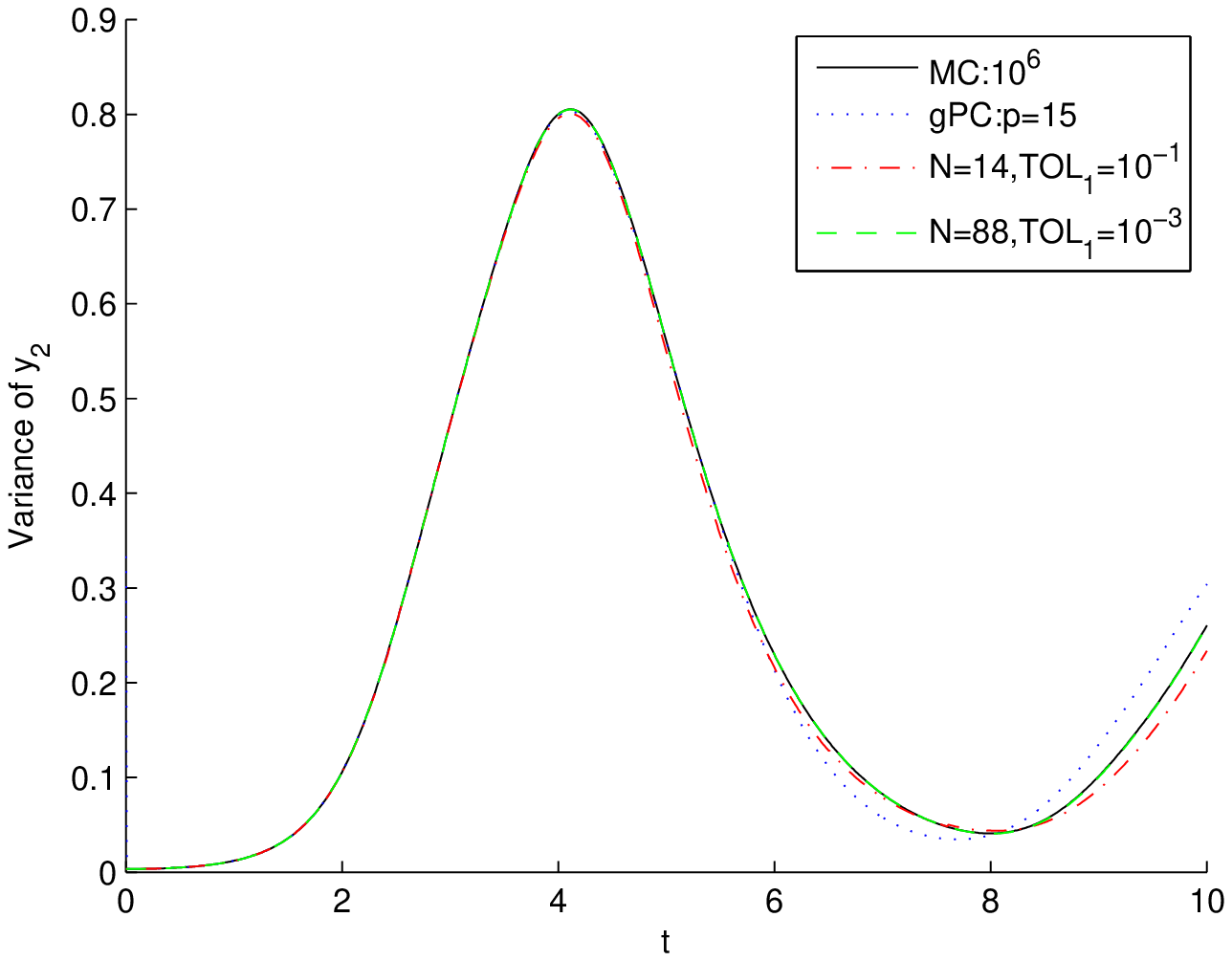}
   \label{y2_var_2d}}
   \subfigure[]{%
   \includegraphics[width = 5.5cm]{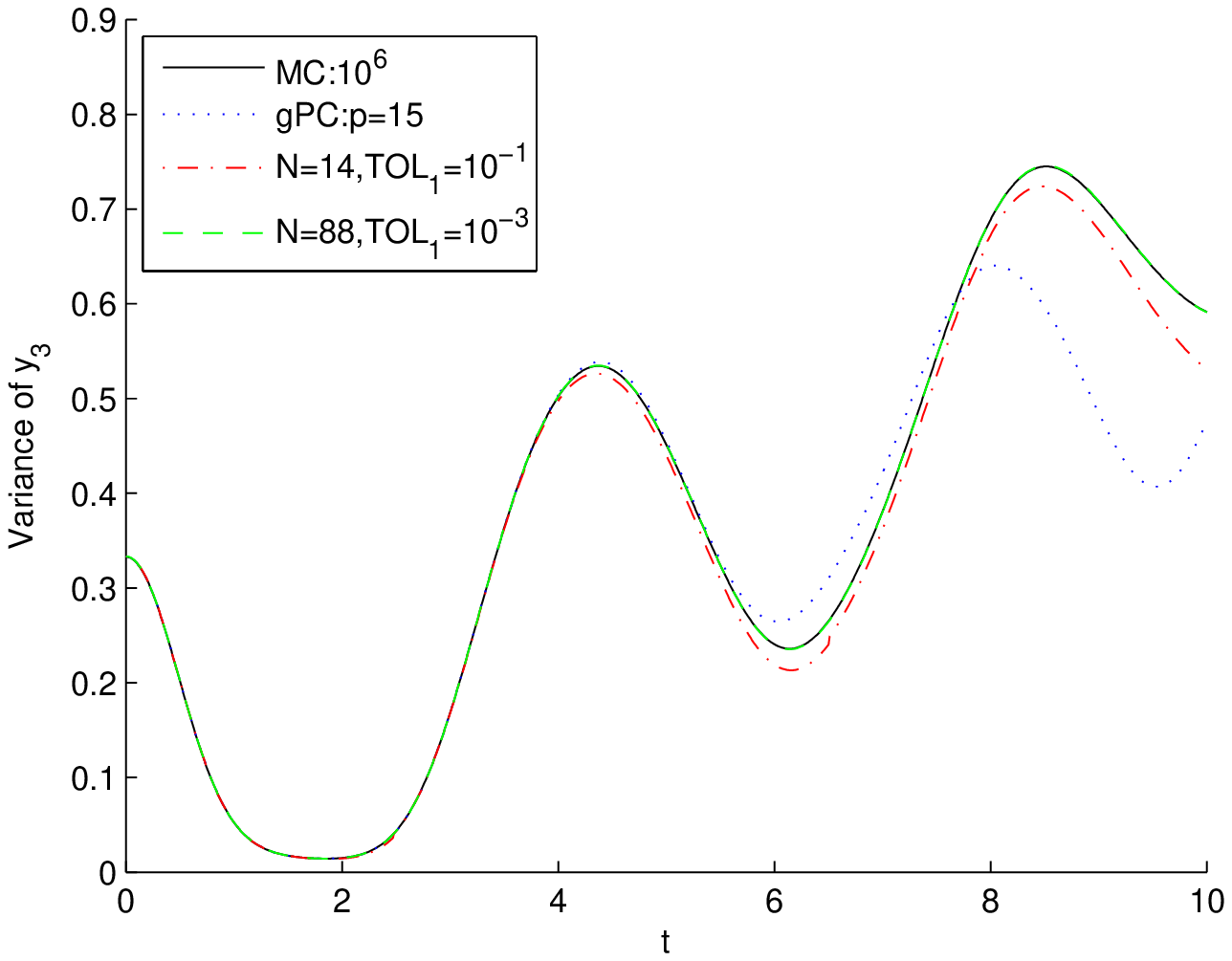}
   \label{y3_var_2d}}
   \subfigure[]{%
   \includegraphics[width = 5.5cm]{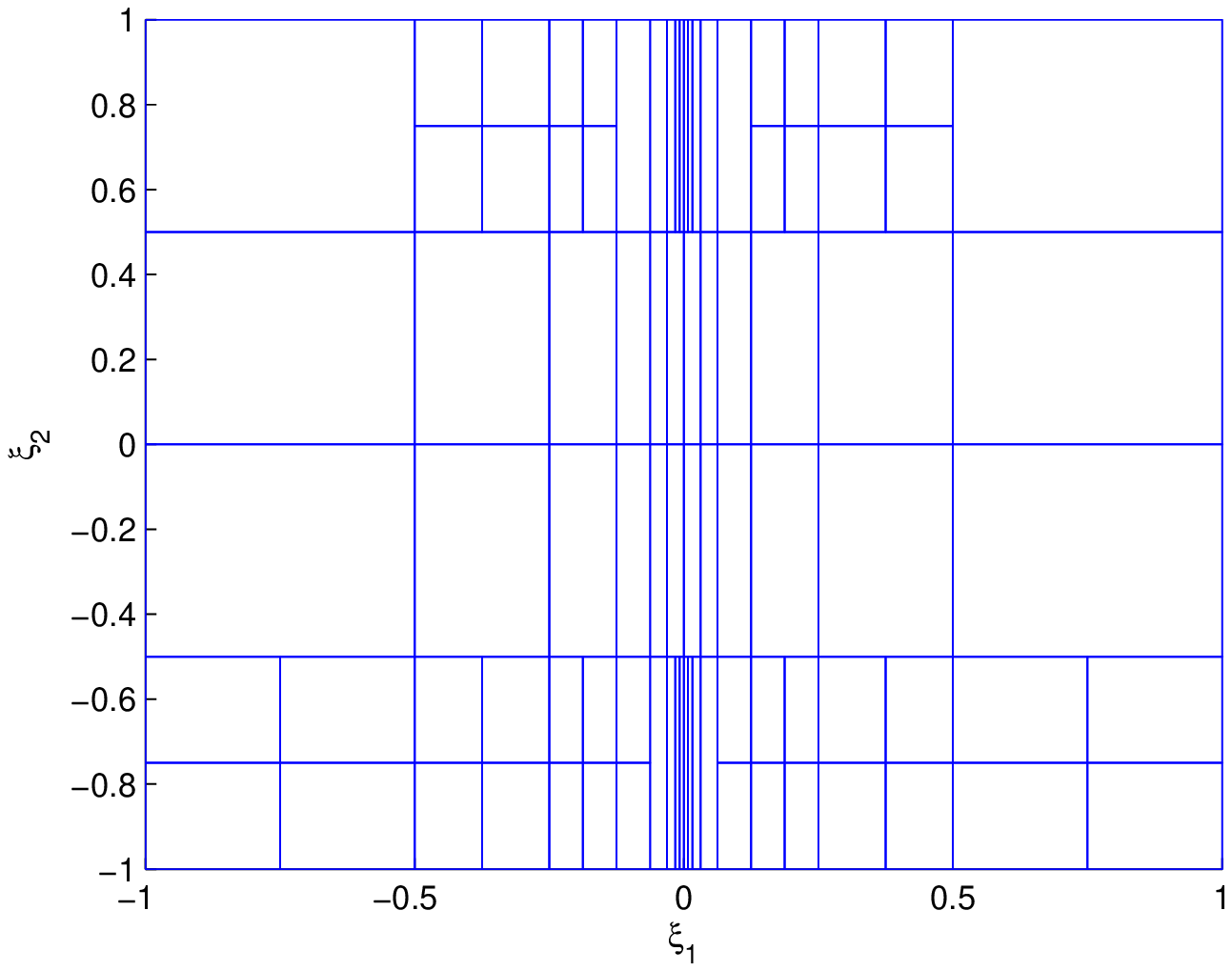}
   \label{KO_mesh_2d}}
\caption{The Kraichnan-Orszag three-mode system with 2D random inputs, $p_r=3,p_f=7,TOL_2=0.1$. (a) Evolution of variance of $y_1$; (b) Evolution of variance of $y_2$; (c) Evolution of variance of $y_3$; (d)Adaptive meshes for 2D random input when $p_r=3, p_f=7, TOL_1=10^{-3}, N=88$.
  }
\label{fig:KO_mesh_2d}
\end{figure}

\begin{table} [htbp]
\begin{center}
\begin{tabular}{llccc}
\toprule
        & N & Error of $var(y_1)$ & Error of $var(y_2)$ & Error of $var(y_3)$\\
\midrule
$TOL_1=10^{-1}$\\
$p_r=2,p_f=5$ & $20$ & $7.0e-2$ & $3.4e-1$ & $3.1e-1$\\
$p_r=3,p_f=7$ & $14$ & $3.2e-2$ & $1.4e-1$ & $1.0e-1$\\
$p_r=4,p_f=7$ & $8$ & $6.6e-2$ & $1.2e-1$ & $4.7e-2$ \\
\midrule
$TOL_1=10^{-3}$\\
$p_r=2,p_f=5$ & $124$ & $1.4e-3$ & $2.8e-2$ & $3.7e-3$\\
$p_r=3,p_f=7$ & $88$ & $3.7e-4$ & $5.2e-3$ & $1.9e-3$\\
$p_r=4,p_f=7$ & $76$ & $7.2e-4$ & $6.6e-3$ & $1.5e-3$\\
\midrule
$TOL_1=10^{-5}$\\
$p_r=2,p_f=5$ & $554$ &  $4.4e-5$ & $1.6e-4$ & $9.0e-5$\\
$p_r=3,p_f=7$ & $304$ & $1.0e-5$ & $3.3e-5$ & $5.7e-5$\\
$p_r=4,p_f=7$ & $262$ &  $9.5e-6$ & $6.0e-5$ & $2.5e-5$\\
\bottomrule
\end{tabular}
\caption{The Kraichnan-Orszag three-mode system with 2D random inputs. Maximum relative errors for the variance of $y_1,y_2$ and $y_3$ when $t\in[0,10].$}
\label{tab:KO_2d}
\end{center}
\end{table}

\subsubsection{Three-dimensional random input}
The initial conditions in this case are
\begin{equation}\label{ex:KO_ini_3d}
y_1(0;\o) = \xi_1(\o), \quad y_2(0;\o) = \xi_2(\o),\quad y_3(0;\omega)=\xi_3(\o),
\end{equation}
where $\xi_1$ and $\xi_2$ are independent uniform random variables on $[-1,1]$. In this case, discontinuities occur at $y_1=0$ and $y_2=0.$ Figure \ref{fig:KO_mesh_3d} shows the evolution of variance of $y_1$ and $y_3$ obtained from different models. The results for a global gPC expansion of order 9 diverges from the Monte Carlo results at $t\approx 3.$ On the other had, our ME-gPC algorithms obtains much better results. We choose results from $p_r=5,p_f=8,TOL_1=10^{-5}$ as reference and show the relative errors in Table \ref{tab:KO_3d}. As we can see, the number of the elements grows dramatically fast in order to gain sufficient accuracy compared to the 1D and 2D cases.
\begin{figure}[htbp]
   \centerline{
   \psfig{file=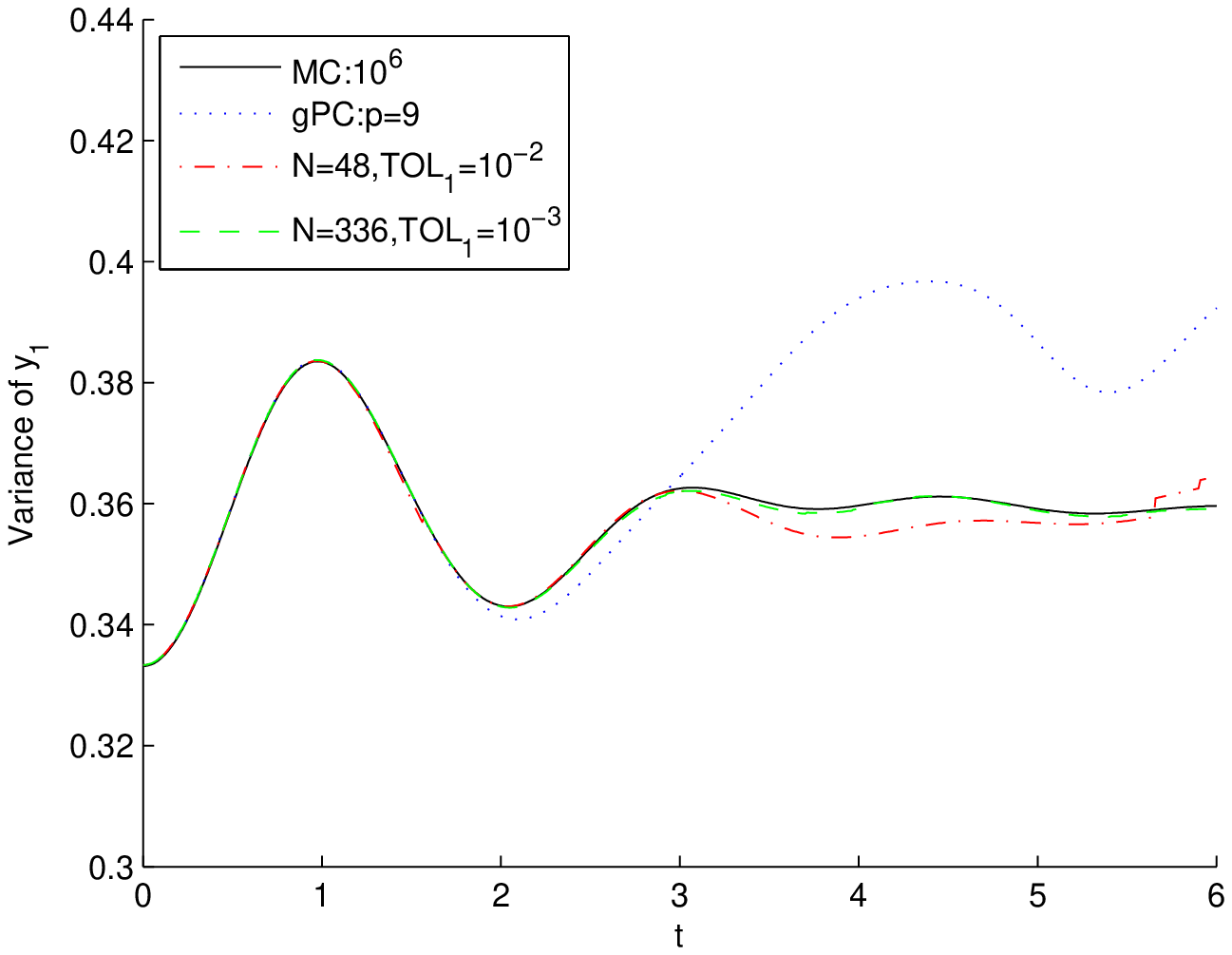,width=7cm}
   \psfig{file=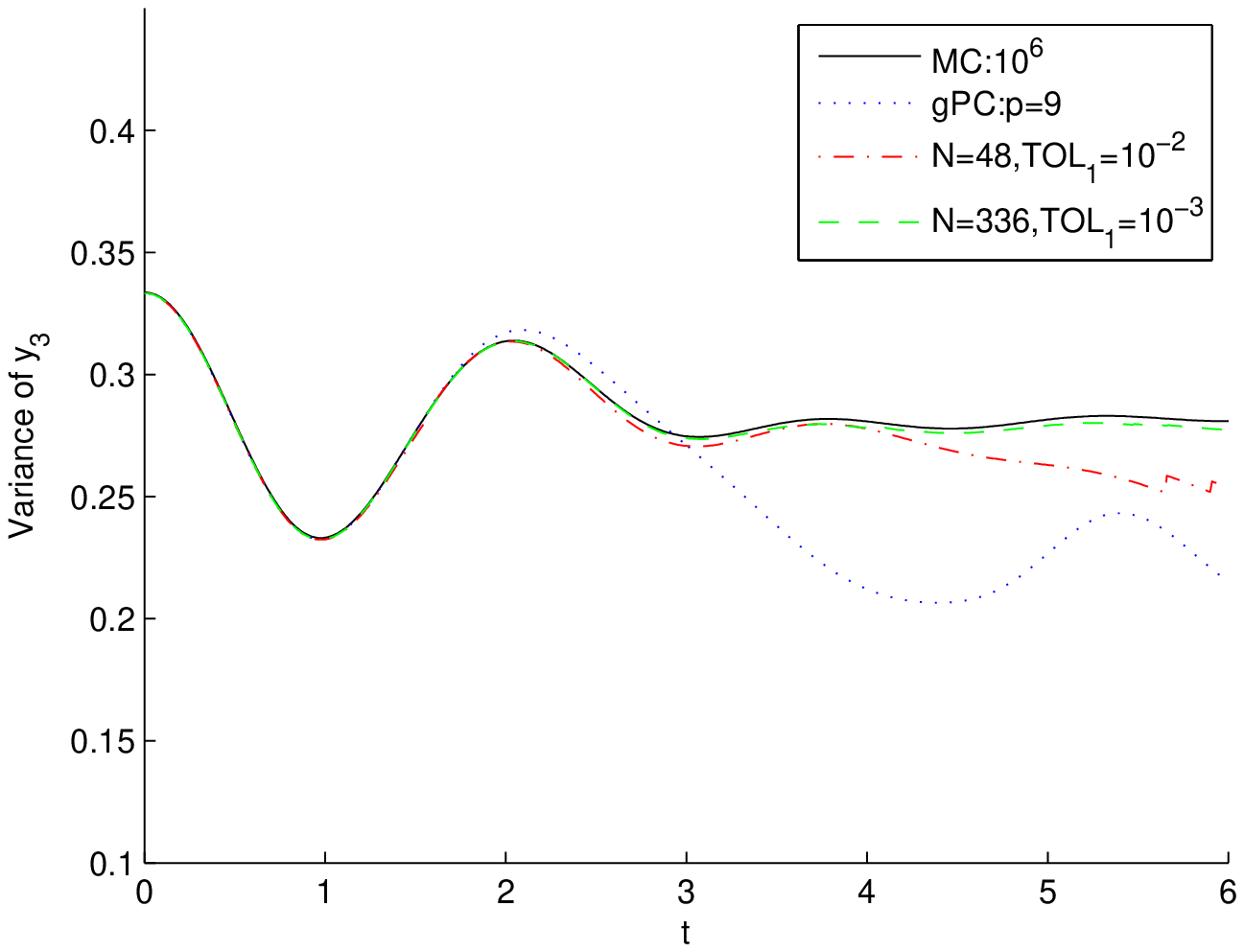,width=7cm}
  }
\caption{Evolution of the variance of $y_1$(left) and evolution of the variance of $y_3$(right) for the Kraichnan-Orszag three-mode system with 3D random inputs.
  }
\label{fig:KO_mesh_3d}
\end{figure}

\begin{table} [htbp]
\begin{center}
\begin{tabular}{llccc}
\toprule
        & N & Error of $var(y_1)$ & Error of $var(y_3)$\\
\midrule
$TOL_1=10^{-2}$\\
$p_r=2,p_f=4$ & $80$ & $1.7e-2$ & $2.4e-1$\\
$p_r=3,p_f=5$ & $48$ & $1.5e-2$ & $1.1e-1$\\
$p_r=4,p_f=7$ & $24$ & $1.2e-2$ & $7.1e-2$\\
\midrule
$TOL_1=10^{-3}$\\
$p_r=2,p_f=4$ & $368$ & $7.1e-3$ & $3.1e-2$ \\
$p_r=3,p_f=5$ & $336$ & $2.7e-3$ & $1.3e-2$ \\
$p_r=4,p_f=7$ & $136$ & $3.2e-3$ & $7.2e-3$\\
\bottomrule
\end{tabular}
\caption{The Kraichnan-Orszag three-mode system with 3D random inputs. Maximum relative errors for the variance of $y_1,y_2$ and $y_3$ when $t\in[0,6].$}
\label{tab:KO_3d}
\end{center}
\end{table}

%%%%%%%%%%%%%%%End of Results%%%%%%%%%%%%%

\section{Discussion and future work}\label{sec:conclusion}
We have presented a novel method for adaptive mesh refinement in the context of uncertainty quantification which is based on model reduction. The main idea behind the proposed approach is that a good reduced model can capture accurately the transfer of activity across scales and thus can be utilized to detect when and where higher resolution is needed.  We have provided theoretical justification as to why this method is appropriate for adaptive mesh refinement. The proposed approach was implemented in the context of multi-element generalized polynomial chaos expansions. The objective was to perform uncertainty quantification in the presence of discontinuities in the random space. The numerical results for the Kraichnan-Orszag system corroborate the theoretical results.

In its current form the proposed method is applicable to problems where the source of randomness is uniformly distributed. For technical reasons, the method is not applicable in its current form to problems with more elaborate random space distributions, for example Gaussianly distributed randomness. However, as  explained in \cite{WanK_JCP05}, one can treat non-uniform sources of randomness by performing an expansion of the non-uniform randomness in a series of uniform random variables. Such a series expansion is equally applicable for our mesh refinement method and results in this direction will be presented in a future publication.

%%%%%%%%%%%%%%%%End of Conclusion%%%%%%%%%%%%%%%

\bibliographystyle{elsarticle-num-names}
\bibliography{random}

\appendix

\section{gPC representation and reduced model of the transformed Kraichnan-Orszag three-mode system}\label{app:gpc}
We use the truncated gPC expansions to approximate the solution of \eqref{ex:KO},
\begin{equation}\label{KO_gpc_sol}
\begin{split}
\hat{y}_1(t,\xi(\o)) &=\sum_{\i\in F\cup G} \hat{y}_{1{\i}}(t)\Phi_\i(\xi(\o)),\\
\hat{y}_2(t,\xi(\o)) &=\sum_{\i\in F\cup G} \hat{y}_{2{\i}}(t)\Phi_\i(\xi(\o)),\\
\hat{y}_3(t,\xi(\o)) &=\sum_{\i\in F\cup G} \hat{y}_{3{\i}}(t)\Phi_\i(\xi(\o)).
\end{split}
\end{equation}
where $F\cup G = \{\i:0\leq|\i|\leq p_{f}\}$, and $F = \{\i:0\leq |\i|\leq p_r\}$, $p_r<p_f$, $p_r,p_f\in\mathbb{N}_0$.
Substitute \eqref{KO_gpc_sol} into \eqref{ex:KO} and perform the Galerkin projection to obtain the system of deterministic ODEs
\begin{equation}\label{KO_galerkin}
\begin{split}
\frac{d\hat{y}_{1\k}}{dt} &= \sum_{\i\in F\cup G}\sum_{\j\in F\cup G}\hat{y}_{1{\i}}\hat{y}_{3{\j}}e_{\i\j\k},\\
\frac{d\hat{y}_{2\k}}{dt} &= -\sum_{\i\in F\cup G}\sum_{\j\in F\cup G}\hat{y}_{2{\i}}\hat{y}_{3{\j}}e_{\i\j\k},\\
\frac{d\hat{y}_{3\k}}{dt} &= \sum_{\i\in F\cup G}\sum_{\j\in F\cup G}(-\hat{y}_{1{\i}}\hat{y}_{1{\j}}+\hat{y}_{2{\i}}\hat{y}_{2{\j}})e_{\i\j\k}, \quad \k\in F\cup G,
\end{split}
\end{equation}
where $e_{\i\j\k} = \int \Phi_\i\Phi_\j\Phi_\k d\P$.
For this system we choose the $t$-model as the reduced model which is given by
\begin{equation}\label{KO_galerkin_rd}
\begin{split}
\frac{d\hat{y}'_{1\k}}{dt} &= \sum_{\i\in F}\sum_{\j\in F}\hat{y}'_{1\i}\hat{y}'_{3\j}e_{\i\j\k}\\
&+t\sum_{\i\in F}\sum_{\j\in G}\sum_{\mathbf{s}\in F}\sum_{\mathbf{t}\in F}(\hat{y}'_{3\i}\hat{y}'_{1\mathbf{s}}\hat{y}'_{3\mathbf{t}}-\hat{y}'_{1\i}\hat{y}'_{1\mathbf{s}}\hat{y}'_{1\mathbf{t}}+\hat{y}'_{1\i}\hat{y}'_{2\mathbf{s}}\hat{y}'_{2\mathbf{t}})e_{\mathbf{s}\mathbf{t}\j}e_{\i\j\k}\\
\frac{d\hat{y}'_{2\k}}{dt} &= -\sum_{\i\in F}\sum_{\j\in F}\hat{y}'_{2\i}\hat{y}'_{3\j}e_{\i\j\k}\\
&+t\sum_{\i\in F}\sum_{\j\in G}\sum_{\mathbf{s}\in F}\sum_{\mathbf{t}\in F}(\hat{y}'_{3\i}\hat{y}'_{2\mathbf{s}}\hat{y}'_{3\mathbf{t}}+\hat{y}'_{2\i}\hat{y}'_{1\mathbf{s}}\hat{y}'_{1\mathbf{t}}-\hat{y}'_{2\i}\hat{y}'_{2\mathbf{s}}\hat{y}'_{2\mathbf{t}})e_{\mathbf{s}\mathbf{t}\j}e_{\i\j\k}\\
\frac{d\hat{y}'_{3\k}}{dt} &= \sum_{\i\in F}\sum_{\j\in F}(-\hat{y}'_{1\i}\hat{y}'_{1\j}+\hat{y}'_{2\i}\hat{y}'_{2\j})e_{\i\j\k}\\
&+t\sum_{\i\in F}\sum_{\j\in G}\sum_{\mathbf{s}\in F}\sum_{\mathbf{t}\in F}(-2\hat{y}'_{1\i}\hat{y}'_{1\mathbf{s}}\hat{y}'_{3\mathbf{t}}-2\hat{y}'_{2\i}\hat{y}'_{2\mathbf{s}}\hat{y}'_{2\mathbf{t}})e_{\mathbf{s}\mathbf{t}\j}e_{\i\j\k}
\end{split}
\end{equation}
For simplicity, we define the following notations:
\begin{equation}\label{function_form_f}
\begin{split}
\hat{R}_{1\k1}^{(0)}(t,\hat{y}_1(t),\hat{y}_2(t),\hat{y}_3(t)) =& \sum_{\i\in F\cup G}\sum_{\j\in F\cup G}\hat{y}_{1\i}\hat{y}_{3\j}e_{\i\j\k}\\
\hat{R}_{1\k2}^{(0)}(t,\hat{y}_1(t),\hat{y}_2(t),\hat{y}_3(t)) =&t\sum_{\i\in F\cup G}\sum_{\j\in I}\sum_{\mathbf{s}\in F\cup G}\sum_{\mathbf{t}\in F\cup G}(\hat{y}_{3\i}\hat{y}_{1\mathbf{s}}\hat{y}_{3\mathbf{t}}-\hat{y}_{1\i}\hat{y}_{1\mathbf{s}}\hat{y}_{1\mathbf{t}}\\
&+
\hat{y}_{1\i}\hat{y}_{2\mathbf{s}}\hat{y}_{2\mathbf{t}})e_{\mathbf{s}\mathbf{t}\j}e_{\i\j\k}\\
\hat{R}_{2\k1}^{(0)}(t,\hat{y}_1(t),\hat{y}_2(t),\hat{y}_3(t)) =& -\sum_{\i\in F\cup G}\sum_{\j\in F\cup G}\hat{y}_{2\i}\hat{y}_{3\j}e_{\i\j\k}\\
\hat{R}_{2\k2}^{(0)}(t,\hat{y}_1(t),\hat{y}_2(t),\hat{y}_3(t)) =&t\sum_{\i\in F\cup G}\sum_{\j\in I}\sum_{\mathbf{s}\in F\cup G}\sum_{\mathbf{t}\in F\cup G}(\hat{y}_{3\i}\hat{y}_{2\mathbf{s}}\hat{y}_{3\mathbf{t}}+\hat{y}_{2\i}\hat{y}_{1\mathbf{s}}\hat{y}_{1\mathbf{t}}\\
&-
\hat{y}_{2\i}\hat{y}_{2\mathbf{s}}\hat{y}_{2\mathbf{t}})e_{\mathbf{s}\mathbf{t}\j}e_{\i\j\k}
\end{split}
\end{equation}
\begin{equation*}%\label{function_form_f}
\begin{split}
\hat{R}_{3\i1}^{(0)}(t,\hat{y}_1(t),\hat{y}_2(t),\hat{y}_3(t)) =&\sum_{\i\in F\cup G}\sum_{\j\in F\cup G}(-\hat{y}_{1\i}\hat{y}_{1\j}+\hat{y}_{2\i}\hat{y}_{2\j})e_{\i\j\k}\\
\hat{R}_{3\k2}^{(0)}(t,\hat{y}_1(t),\hat{y}_2(t),\hat{y}_3(t)) =&t\sum_{\i\in F\cup G}\sum_{\j\in I}\sum_{\mathbf{s}\in F\cup G}\sum_{\mathbf{t}\in F\cup G}(-2\hat{y}_{1\i}\hat{y}_{1\mathbf{s}}\hat{y}_{3\mathbf{t}}\\
&-2\hat{y}_{2\i}\hat{y}_{2\mathbf{s}}
\hat{y}_{2\mathbf{t}})e_{\mathbf{s}\mathbf{t}\j}e_{\i\j\k}
\end{split}
\end{equation*}
and
\begin{equation}\label{function_form_r}
\begin{split}
\hat{R}_{1\k1}^{(1)}(t,\hat{y}'_1(t),\hat{y}'_2(t),\hat{y}'_3(t)) =&\sum_{\i\in F}\sum_{\j\in F}\hat{y}'_{1\i}\hat{y}'_{3\j}e_{\i\j\k}\\
\hat{R}_{1\k2}^{(1)}(t,\hat{y}'_1(t),\hat{y}'_2(t),\hat{y}'_3(t)) =&t\sum_{\i\in F}\sum_{\j\in G}\sum_{\mathbf{s}\in F}\sum_{\mathbf{t}\in F}(\hat{y}'_{3\i}\hat{y}'_{1\mathbf{s}}\hat{y}'_{3\mathbf{t}}-\hat{y}'_{1\i}\hat{y}'_{1\mathbf{s}}\hat{y}'_{1\mathbf{t}}\\
&+\hat{y}'_{1\i}\hat{y}'_{2\mathbf{s}}\hat{y}'_{2\mathbf{t}})e_{\mathbf{s}\mathbf{t}\j}e_{\i\j\k}\\
\hat{R}_{2\k1}^{(1)}(t,\hat{y}'_1(t),\hat{y}'_2(t),\hat{y}'_3(t)) =&-\sum_{\i\in F}\sum_{\j\in F}\hat{y}'_{2\i}\hat{y}'_{3\j}e_{\i\j\k}\\
\hat{R}_{2\k2}^{(1)}(t,\hat{y}'_1(t),\hat{y}'_2(t),\hat{y}'_3(t)) =&t\sum_{\i\in F}\sum_{\j\in G}\sum_{\mathbf{s}\in F}\sum_{\mathbf{t}\in F}(\hat{y}'_{3\i}\hat{y}'_{2\mathbf{s}}\hat{y}'_{3\mathbf{t}}+\hat{y}'_{2\i}\hat{y}'_{1\mathbf{s}}\hat{y}'_{1\mathbf{t}}\\
&-\hat{y}'_{2\i}\hat{y}'_{2\mathbf{s}}\hat{y}'_{2\mathbf{t}})e_{\mathbf{s}\mathbf{t}\j}e_{\i\j\k}\\
\hat{R}_{3\k1}^{(1)}(t,\hat{y}'_1(t),\hat{y}'_2(t),\hat{y}'_3(t)) =&\sum_{\i\in F}\sum_{\j\in F}(-\hat{y}'_{1\i}\hat{y}'_{1\j}+\hat{y}'_{2\i}\hat{y}'_{2\j})e_{\i\j\k}\\
\hat{R}_{3\k2}^{(1)}(t,\hat{y}'_1(t),\hat{y}'_2(t),\hat{y}'_3(t)) =&t\sum_{\i\in F}\sum_{\j\in G}\sum_{\mathbf{s}\in F}\sum_{\mathbf{t}\in F}(-2\hat{y}'_{1\i}\hat{y}'_{1\mathbf{s}}\hat{y}'_{3\mathbf{t}}\\
&-2\hat{y}'_{2\i}\hat{y}'_{2\mathbf{s}}\hat{y}'_{2\mathbf{t}})e_{\mathbf{s}\mathbf{t}\j}e_{\i\j\k}
\end{split}
\end{equation}
Note that $R_{i\k j}^{(0)}$ have the same functional form as $R_{i\k j}^{(1)}$, $i=1,2,3$, $j=1,2.$
%Let $\mathbf{y} = (\mathbf{y}_1,\mathbf{y}_2,\mathbf{y}_3)^T$, and $\mathbf{R}_1^{(0)} = %(\hat{R}_{1\k1}^{(0)},\hat{R}_{2\k1}^{(0)},\hat{R}_{3\k1}^{(0)})^T$, $\mathbf{R}_2^{(0)} = %(\hat{R}_{1\k2}^{(0)},\hat{R}_{2\k2}^{(0)},\hat{R}_{3\k2}^{(0)})^T$, $\mathbf{R}_1^{(1)} = %(\hat{R}_{1\k1}^{(1)},\hat{R}_{2\k1}^{(1)},\hat{R}_{3\k1}^{(1)})^T$, and $\mathbf{R}_2^{(1)} = %(\hat{R}_{1\k2}^{(1)},\hat{R}_{2\k2}^{(1)},\hat{R}_{3\k2}^{(1)})^T$ .
Then the full model \eqref{KO_galerkin} and the reduced model \eqref{KO_galerkin_rd} can be rewritten as
\begin{equation}
\begin{split}
\frac{d\hat{y}_{1\k}}{dt} &= \sum_{i=1}^2 a_{1i}^{(0)}\hat{R}^{(0)}_{1\k i}(t,\hat{y}_1,\hat{y}_2,\hat{y}_3),\\
\frac{d\hat{y}_{2\k}}{dt} &= \sum_{i=1}^2a_{2i}^{(0)}\hat{R}^{(0)}_{2\k i}(t,\hat{y}_1,\hat{y}_2,\hat{y}_3),\\
\frac{d\hat{y}_{3\k}}{dt} &= \sum_{i=1}^2a_{3i}^{(0)}\hat{R}^{(0)}_{3\k i}(t,\hat{y}_1,\hat{y}_2,\hat{y}_3),\quad \k\in{F\cup G}
\end{split}
\end{equation}
where $a_{11}^{(0)} = 1$, $a_{12}^{(0)} = 0$, $a_{21}^{(0)} = 1$, $a_{22}^{(0)} = 0$, $a_{31}^{(0)} = 1$, $a_{32}^{(0)} = 0$, and
\begin{equation}
\begin{split}
\frac{d\hat{y}'_{1\k}}{dt} &= \sum_{i=1}^2 a_{1i}^{(1)}\hat{R}^{(1)}_{1\k i}(t,\hat{y}'_1,\hat{y}'_2,\hat{y}'_3),\\
\frac{d\hat{y}'_{2\k}}{dt} &= \sum_{i=1}^2a_{2i}^{(1)}\hat{R}^{(1)}_{2\k i}(t,\hat{y}'_1,\hat{y}'_2,\hat{y}'_3),
\end{split}
\end{equation}
\begin{equation*}
\begin{split}
\frac{d\hat{y}'_{3\k}}{dt} &= \sum_{i=1}^2a_{3i}^{(1)}\hat{R}^{(1)}_{3\k i}(t,\hat{y}'_1,\hat{y}'_2,\hat{y}'_3),\quad \k\in{F}
\end{split}
\end{equation*}
where $a_{11}^{(1)} = 1$, $a_{12}^{(1)} = 1$, $a_{21}^{(1)} = 1$, $a_{22}^{(1)} = 1$, $a_{31}^{(1)} = 1$, $a_{32}^{(1)} = 1$.

The goal of our adaptive mesh refinement approach is to capture the statistical properties of the solution. We have chosen as a criterion for mesh refinement the rate of change  of  $\hat{E} = \sum_{\i\in F}|y_{1\i}|^2+\sum_{\i\in F}|y_{2\i}|^2+\sum_{\i\in F}|y_{3\i}|^2.$

The rate of change of $\hat{E}$ for the full model is given by
\begin{equation}
\begin{split}
\frac{d\hat{E}}{dt}=&2\sum_{\k\in F}\sum_{i=1}^3a_{i1}^{(0)}\Re\left(\hat{R}_{i\k 1}^{(0)}(t,\hat{y}_1,\hat{y}_2,\hat{y}_3)\hat{y}^*_{i\k}\right)\\
&+2\sum_{\k\in F}\sum_{i=1}^3a_{i2}^{(0)}\Re\left(\hat{R}_{i\k 2}^{(0)}(t,\hat{y}_1,\hat{y}_2,\hat{y}_3)\hat{y}^*_{i\k}\right),
%\frac{d\hat{E}_2}{dt}=&4\sum_{\k\in F}\sum_{i=1}^3a_{i1}^{(0)}\Re\left(\hat{R}_{i\k %1}^{(0)}(t,\hat{y}_1,\hat{y}_2,\hat{y}_3)|\hat{y}|^2_{i\k}\hat{y}^*_{i\k}\right)\\
%&+4\sum_{\k\in F}\sum_{i=1}^3a_{i2}^{(0)}\Re\left(\hat{R}_{i\k %2}^{(0)}(t,\hat{y}_1,\hat{y}_2,\hat{y}_3)|\hat{y}_{i\k}|^2\hat{y}^*_{i\k}\right),
\end{split}
\end{equation}
where $(\hat{y}_{i\k})^*$ is the complex conjugate of $\hat{y}_{i\k}.$ The rate of change of $$\hat{E'} = \sum_{\i\in F}|y'_{1\i}|^2+\sum_{\i\in F}|y'_{2\i}|^2+\sum_{\i\in F}|y'_{3\i}|^2$$ for the reduced model is given by
\begin{equation}
\begin{split}
\frac{d\hat{E'}}{dt}=&2\sum_{\k\in F}\sum_{i=1}^3a_{i1}^{(1)}\Re\left(\hat{R}_{i\k 1}^{(0)}(t,\hat{y}'_1,\hat{y}'_2,\hat{y}'_3)(\hat{y}'_{i\k})^*\right)\\
&+2\sum_{\k\in F}\sum_{i=1}^3a_{i2}^{(0)}\Re\left(\hat{R}_{i\k 2}^{(0)}(t,\hat{y}'_1,\hat{y}'_2,\hat{y}'_3)(\hat{y}'_{i\k})^*\right),
%\frac{d\hat{E}_2}{dt}=&4\sum_{\k\in F}\sum_{i=1}^3a_{i1}^{(1)}\Re\left(\hat{R}_{i\k %1}^{(0)}(t,\hat{y}'_1,\hat{y}'_2,\hat{y}'_3)|\hat{y}'_{i\k}|^2(\hat{y}'_{i\k})^*\right)\\
%&+4\sum_{\k\in F}\sum_{i=1}^3a_{i2}^{(0)}\Re\left(\hat{R}_{i\k %2}^{(0)}(t,\hat{y}'_1,\hat{y}'_2,\hat{y}'_3)|\hat{y}'_{i\k}|^2(\hat{y}'_{i\k})^*\right),
\end{split}
\end{equation}
where $(\hat{y}'_{i\k})^*$ is the complex conjugate of $\hat{y}'_{i\k}.$

\section{Error of the $t$-model for the Kraichnan-Orszag three-mode system}\label{app:pro}
We will show the relation between the rate of change of $\hat{E'} = \sum_{\i\in F}|y'_{1\i}|^2+\sum_{\i\in F}|y'_{2\i}|^2+\sum_{\i\in F}|y'_{3\i}|^2=\|\mb{\hat{y}'}\|_{L_2(\O)}^2$ and the error of the $t$-model. Use the notations from Appendix. \ref{app:gpc}, and let $\mb{P}$ be the projection onto the space spanned by $\{\Phi_{\i}|\i\in F\}$.
Let $\mb{B}(y_i,y_j) = y_iy_j$, $i,j=1,2,3.$ Then, the Kraichnan-Orszag three-mode system \eqref{ex:KO} can be written as
\begin{equation}\label{ex:KO_r}
\frac{dy_1}{dt} = \mb{B}(y_1,y_3),\quad \frac{dy_2}{dt} = -\mb{B}(y_2,y_3),\quad \frac{dy_3}{dt} = -\mb{B}(y_1,y_1)+\mb{B}(y_2,y_2).
\end{equation}
and its projection
\begin{equation}\label{ex:KO_p}
\begin{split}
\frac{d\mb{P}y_1}{dt} &= \mb{PB}(y_1,y_3),\\
\frac{d\mb{P}y_2}{dt} &= -\mb{PB}(y_2,y_3),\\
\frac{d\mb{P}y_3}{dt} &= -\mb{PB}(y_1,y_1)+\mb{PB}(y_2,y_2).
\end{split}
\end{equation}
The $t$-model can be written as
\begin{equation}\label{ex:KO_t}
\begin{split}
\frac{d\hat{y}_1'}{dt} =& \mb{PB}(\hat{y}_1',\hat{y}_3')+t\mb{P}\left\{\mb{B}((\mb{I}-\mb{P})\mb{B}(\hat{y}_1',\hat{y}_3'),\hat{y}_3')\right\}\\
&+t\mb{P}\left\{\mb{B}\left(\hat{y}_1',-(\mb{I}-\mb{P})\mb{B}(\hat{y}_1',\hat{y}_1')+(\mb{I}-\mb{P})\mb{B}(\hat{y}_2',\hat{y}_2')\right)\right\},\\
\frac{d\hat{y}_2'}{dt} =& -\mb{PB}(\hat{y}_2',\hat{y}_3')-t\mb{P}\left\{\mb{B}(-(\mb{I}-\mb{P})\mb{B}(\hat{y}_2',\hat{y}_3'),\hat{y}_3')\right\}\\
&-t\mb{P}\left\{\mb{B}(\hat{y}_2',-(\mb{I}-\mb{P})\mb{B}(\hat{y}_1',\hat{y}_1')+(\mb{I}-\mb{P})\mb{B}(\hat{y}_2',\hat{y}_2'))\right\},\\
\frac{d\hat{y}_3'}{dt} =& -\mb{PB}(\hat{y}_1',\hat{y}_1')+\mb{PB}(\hat{y}_2',\hat{y}_2')\\
&-t\mb{P}\left\{\mb{B}((\mb{I}-\mb{P})\mb{B}(\hat{y}_1',\hat{y}_3'),\hat{y}_1')
+\mb{B}(\hat{y}_1',-(\mb{I}-\mb{P})\mb{B}(\hat{y}_1',\hat{y}_3'))\right\}\\
&+t\mb{P}\left\{\mb{B}(-(\mb{I}-\mb{P})\mb{B}(\hat{y}_2',\hat{y}_3'),\hat{y}_2')
+\mb{B}(\hat{y}_2',-(\mb{I}-\mb{P})\mb{B}(\hat{y}_2',\hat{y}_3'))\right\}.
\end{split}
\end{equation}

\begin{theorem}\label{le}
Let $\mb{\hat{y}}' = (\hat{y}_1',\hat{y}_2',\hat{y}_3')^{T}$ and $\mb{\Gamma} =(\Gamma_1,\Gamma_2,\Gamma_3)^T$ where
$\Gamma_1 = (\mb{I}-\mb{P})\mb{B}(\hat{y}_1',\hat{y}_3')$, $\Gamma_2 =-(\mb{I}-\mb{P})\mb{B}(\hat{y}_2',\hat{y}_3')$, $\Gamma_3 = -(\mb{I}-\mb{P})\mb{B}(\hat{y}_1',\hat{y}_1')+(\mb{I}-\mb{P})\mb{B}(\hat{y}_2',\hat{y}_2'),$ then
\begin{equation}\label{rate_of_change}
\frac{1}{2}\frac{d}{dt}\|\mb{\hat{y}'}\|_{L_2(\O)}^2 = -t\|\mb{\Gamma}\|_{L_2(\O)}^2,
\end{equation}
\end{theorem}
\begin{proof}For simplicity, we use $(\cdot,\cdot)$ to denote the inner product and $\|\cdot\|$ to denote the $L_2$ norm on the random space. From \eqref{ex:KO_t}, we obtain
\begin{equation}\label{ex:KO_t_norm}
\begin{split}
\frac{1}{2}\frac{d}{dt}\|\hat{y}_1'\|^2 =& \Big(\mb{B}(\hat{y}_1',\hat{y}_3'),\hat{y}_1'\Big)+t\Big(\mb{B}(\Gamma_1,\hat{y}_3'),\hat{y}_1'\Big)+t\Big(\mb{B}(\hat{y}_1',\Gamma_3),\hat{y}_1'\Big),\\
\frac{1}{2}\frac{d}{dt}\|\hat{y}_2'\|^2 =& -\Big(\mb{B}(\hat{y}_2',\hat{y}_3'),\hat{y}_2'\Big)-t\Big(\mb{B}(\Gamma_2,\hat{y}_3'),\hat{y}_2'\big)-t\Big(\mb{B}(\hat{y}_2',\Gamma_3),\hat{y}_2'\big),\\
\frac{1}{2}\frac{d}{dt}\|\hat{y}_3'\|^2 =& \Big(-\mb{B}(\hat{y}_1',\hat{y}_1')+\mb{B}(\hat{y}_2',\hat{y}_2'),\hat{y}_3'\Big)\\
&-t\Big(\mb{B}(\Gamma_1,\hat{y}_1')
+\mb{B}(\hat{y}_1',\Gamma_1),\hat{y}_3'\Big)+t\Big(\mb{B}(\Gamma_2,\hat{y}_2')+\mb{B}(\hat{y}_2',\Gamma_2),\hat{y}_3'\Big).
\end{split}
\end{equation}
First, we claim that
\begin{equation}\label{inner_prod}
\Big(\mb{B}(f,g),h\Big) = \Big(\mb{B}(f,h),g\Big) = \Big(\mb{B}(g,h),f\Big).
\end{equation}
To show \eqref{inner_prod}, we assume $f = \sum_{\i}f_\i\Phi_\i$, $g = \sum_{\i}g_\i\Phi_\i$, $h = \sum_{\i}h_\i\Phi_\i.$ We find
\[
\begin{split}
\Big(\mb{B}(f,g),h\Big) &= \Big(\mb{B}(\sum_{\i}f_\i\Phi_\i,\sum_{\j}g_\j\Phi_\j),\sum_{\k}h_\k\Phi_\k \Big)\\
&=\sum_\i\sum_\j\sum_\k f_\i g_\j h_\k\int_{\O}\Phi_\i\Phi_\j\Phi_\k d\P.
\end{split}
\]
Obviously, it can be verified that
\[
\begin{split}
\Big(\mb{B}(f,h),g\Big) =& \sum_\i\sum_\j\sum_\k f_\i g_\j h_\k\int_{\O}\Phi_\i\Phi_\j\Phi_\k d\P,\\
\Big(\mb{B}(g,h),f\Big) = &\sum_\i\sum_\j\sum_\k f_\i g_\j h_\k\int_{\O}\Phi_\i\Phi_\j\Phi_\k d\P.
\end{split}
\]
Consequently, \eqref{inner_prod} is satisfied.
Given the fact that $\mb{P}\perp (\mb{I}-\mb{P})$, and \eqref{inner_prod}, we have
\begin{equation*}
\begin{split}
\frac{1}{2}\frac{d}{dt}\|\mb{\hat{y}'}\|^2 =
&t\Big(\mb{B}(\hat{y}_1',\Gamma_3),\hat{y}_1'\Big)-t\Big(\mb{B}(\hat{y}_2',\Gamma_3),\hat{y}_2'\big)\\
&-t\Big(\mb{B}(\Gamma_1,\hat{y}_1'),\hat{y}_3'\Big)+t\Big(\mb{B}(\hat{y}_2',\Gamma_2),\hat{y}_3'\Big)\\
=&t\Big(\mb{B}(\hat{y}_1',-(\mb{I}-\mb{P})\mb{B}(\hat{y}_1',\hat{y}_1')+(\mb{I}-\mb{P})\mb{B}(\hat{y}_2',\hat{y}_2')),\hat{y}_1'\Big)\\
&-t\Big(\mb{B}(\hat{y}_2',-(\mb{I}-\mb{P})\mb{B}(\hat{y}_1',\hat{y}_1')+(\mb{I}-\mb{P})\mb{B}(\hat{y}_2',\hat{y}_2')),\hat{y}_2'\big)\\
&-t\Big(\mb{B}(\Gamma_1,\hat{y}_1'),\hat{y}_3'\Big)+t\Big(\mb{B}(\hat{y}_2',\Gamma_2),\hat{y}_3'\Big)\\
=&-t\Big(\mb{B}(\hat{y}_1',\hat{y}_1'),(\mb{I}-\mb{P})\mb{B}(\hat{y}_1',\hat{y}_1')\Big)+t\Big(\mb{B}(\hat{y}_2',\hat{y}_2'),(\mb{I}-\mb{P})\mb{B}(\hat{y}_1',\hat{y}_1')\Big)\\
&+t\Big(\mb{B}(\hat{y}_2'\hat{y}_2'),(\mb{I}-\mb{P})\mb{B}(\hat{y}_1',\hat{y}_1')\Big)-t\Big(\mb{B}(\hat{y}_2',\hat{y}_2'),(\mb{I}-\mb{P})\mb{B}(\hat{y}_2',\hat{y}_2')\Big)\\
&-t\Big(\mb{B}(\Gamma_1,\hat{y}_1'),y_3'\Big)+t\Big(\mb{B}(\hat{y}_2',\Gamma_2),\hat{y}_3'\Big)\\
=&-t\|(\mb{I}-\mb{P})\mb{B}(\hat{y}_1',\hat{y}_1')\|^2-t\|(\mb{I}-\mb{P})\mb{B}(\hat{y}_2',\hat{y}_2')\|^2\\
&+2t\Big((\mb{I}-\mb{P})\mb{B}(\hat{y}_1',\hat{y}_1'),(\mb{I}-\mb{P})\mb{B}(\hat{y}_2',\hat{y}_2')\Big)\\
&-t\Big(\mb{B}(\hat{y}_1',\hat{y}_3'),\Gamma_1\Big)+t\Big(\mb{B}(\hat{y}_2',\hat{y}_3'),\Gamma_2\Big)\\
=&-t\|(\mb{I}-\mb{P})\mb{B}(\hat{y}_1',\hat{y}_1')-(\mb{I}-\mb{P})\mb{B}(\hat{y}_2',\hat{y}_2')\|^2\\
&-t\Big(\mb{B}(\hat{y}_1',\hat{y}_3'),(\mb{I}-\mb{P})\mb{B}(\hat{y}_1',\hat{y}_3')\Big)-t\Big(-\mb{B}(\hat{y}_2',\hat{y}_3'),-(\mb{I}-\mb{P})\mb{B}(\hat{y}_2',\hat{y}_3')\Big)\\
=&-t\|\Gamma_3\|^2-t\|\Gamma_1\|^2-t\|\Gamma_2\|^2.
\end{split}
\end{equation*}
\end{proof}

With the same notations as before, we have the following theorem which characterizes the error of the $t$-model system.
\begin{theorem}\label{le_2}
Let $\mb{y} = (y_1,y_2,y_3)^T$, and $\mb{P}\mb{y} = (\mb{P}y_1,\mb{P}y_2,\mb{P}y_3)^{T}$, where $y_1,y_2,y_3$ satisfy \eqref{ex:KO}. Then, there exist constants $A,B$, and $C$ such that
\begin{equation}\label{KO_error}
\begin{split}
\frac{1}{2}\frac{d}{dt}\|\mb{\hat{y}}'-\mb{P}\mb{y}\|^2 \leq &(A+tB)\|\hat{\mb{y}}'-\mb{P}\mb{y}\|^2+C\|(\mb{I}-\mb{P})\mb{y}\|^2\\
&+5t\|(\mb{I}-\mb{P})\mb{B}(\mb{P}y_1,\mb{P}y_3)\|^2+5t\|(\mb{I}-\mb{P})\mb{B}(\mb{P}y_2,\mb{P}y_3)\|^2\\
&+t\|(\mb{I}-\mb{P})\mb{B}(\mb{P}y_1,\mb{P}y_1)\|^2+t\|(\mb{I}-\mb{P})\mb{B}(\mb{P}y_2,\mb{P}y_2)\|^2
\end{split}
\end{equation}
\end{theorem}
\begin{proof} The left side of \eqref{KO_error} can be expressed as
\begin{equation}\label{KO_error_1}
\begin{split}
\frac{1}{2}\frac{d}{dt}\|\mb{\hat{y}}'-\mb{P}\mb{y}\|^2 =& \frac{1}{2}\frac{d}{dt}\Big(\mb{\hat{y}}'-\mb{Py},\mb{\hat{y}}'-\mb{Py}\Big)\\
=&\frac{1}{2}\sum_{i=1}^3\Big(\hat{y}_i'-\mb{P}y_i,\frac{d}{dt}(\hat{y}_i'-\mb{P}y_i)\Big).
\end{split}
\end{equation}
For $\hat{y}_1'-\mb{P}y_1$,
\begin{equation}
\begin{split}
\Big(\hat{y}_1'-\mb{P}y_1, \frac{d}{dt}(\hat{y}_1'-\mb{P}y_1)\Big) &= \Big(\hat{y}_1'-\mb{P}y_1, \mb{P}\mb{B}(\hat{y}_1',\hat{y}_3')-\mb{P}\mb{B}(y_1,y_3)\Big) \\
&+\Big(\hat{y}_1'-\mb{P}y_1, t\mb{P}\{\mb{B}(\Gamma_1,\hat{y}_3')+\mb{B}(\hat{y}_1',\Gamma_3)\}\Big)
\end{split}
\end{equation}
where $\hat{\mb{y}}'$ satisfies \eqref{ex:KO_t} and $\mb{Py}$ satisfies \eqref{ex:KO_p}.
Let $I_1 = \Big(\hat{y}_1'-\mb{P}y_1, \mb{PB}(\hat{y}_1',\hat{y}_3')-\mb{PB}(y_1,y_3)\Big)$, $I_2 = \Big(\hat{y}_1', t\mb{P}\{\mb{B}(\Gamma_1,\hat{y}_3')+\mb{B}(\hat{y}_1',\Gamma_3)\}\Big)$ and $I_3 = -\Big(\mb{P}y_1, t\mb{P}\{\mb{B}(\Gamma_1,\hat{y}_3')+\mb{B}(\hat{y}_1',\Gamma_3)\}\Big)$.
Since $\mb{P}$ is self-adjoint and $\mb{B}$ is continuous, it follows that
\begin{equation}\label{I_1}
\begin{split}
I_1 = &  \Big(\hat{y}_1'-\mb{P}y_1, \mb{B}(\hat{y}_1',\hat{y}_3')-\mb{B}(y_1,y_3)\Big)\\
\leq & \|\hat{y}_1'-\mb{P}y_1\|\|\mb{B}(\hat{y}_1',\hat{y}_3')-\mb{B}(y_1,y_3)\|\\
\leq & \|\hat{\mb{y}}-\mb{Py}\|C_1\|\hat{\mb{y}}-\mb{y}\|\\
\leq & C_1 \|\hat{\mb{y}}'-\mb{Py}\|(\|\hat{\mb{y}}'-\mb{Py}\|+\|(\mb{I}-\mb{P})\mb{y}\|)\\
= & C_1 \|\hat{\mb{y}}'-\mb{Py}\|^2+C_1\|\hat{\mb{y}}'-\mb{Py}\|\|(\mb{I}-\mb{P})\mb{y}\|,
\end{split}
\end{equation}
where $C_1$ is some constant. Also, for $I_3$ we have
\begin{equation}\label{I_3}
\begin{split}
I_3 =& -t\Big(\mb{P}y_1, \mb{P}\{\mb{B}(\Gamma_1,\hat{y}_3')+\mb{B}(\hat{y}_1',\Gamma_3)\}\Big)\\
=&-t\Big(\mb{P}y_1,\mb{P}\mb{B}(\Gamma_1,\hat{y}_3'-\mb{P}y_3)\Big)-t\Big(\mb{P}y_1,\mb{P}\mb{B}(\Gamma_1,\mb{P}y_3)\Big)\\
&-t\Big(\mb{P}y_1,\mb{P}\mb{B}(\Gamma_3,\hat{y}_1'-\mb{P}y_1)\Big)-t\Big(\mb{P}y_1,\mb{P}\mb{B}(\Gamma_3,\mb{P}y_1)\Big)\\
=&-t\Big(\mb{P}y_1'(\hat{y}_3'-\mb{P}y_3),\Gamma_1\Big)-\Big(\mb{P}y_1\mb{P}y_3, \Gamma_1\Big)\\
&-t\Big(\mb{P}y_1'(\hat{y}_1'-\mb{P}y_1),\Gamma_3\Big)-\Big(\mb{P}y_1\mb{P}y_1, \Gamma_3\Big)\\
=&-t\Big(\mb{P}y_1'(\hat{y}_3'-\mb{P}y_3),\Gamma_1\Big)-t\Big((\mb{I}-\mb{P})\mb{P}y_1\mb{P}y_3,\Gamma_1\Big)\\
&-t\Big(\mb{P}y_1(\hat{y}_1'-\mb{P}y_1),\Gamma_3\Big)-t\Big((\mb{I}-\mb{P})(\mb{P}y_1)^2,\Gamma_3\Big)\\
\leq& t\|\mb{P}y_1(\hat{y}_3'-\mb{P}y_3)\|\|\Gamma_1\|+t\|(\mb{I}-\mb{P})(\mb{P}y_1\mb{P}y_3)\|\|\Gamma_1\|\\
&+t\|\mb{P}y_1(\hat{y}_1'-\mb{P}y_1)\|\|\Gamma_3\|+t\|(\mb{I}-\mb{P})(\mb{P}y_1)^2\|\|\Gamma_3\|\\\leq&t\|\mb{P}y_1\|^2_{L^{\infty}}\|\hat{y}_3'-\mb{P}y_3\|^2+\frac{t}{4}\|\Gamma_1\|^2\\
&+t\|(\mb{I}-\mb{P})(\mb{P}y_1\mb{P}y_3)\|^2+\frac{t}{4}\|\Gamma_1\|^2\\
&+t\|\mb{P}y_1\|^2_{L^{\infty}}\|\hat{y}_1'-\mb{P}y_1\|^2+\frac{t}{4}\|\Gamma_3\|^2\\
&+t\|(\mb{I}-\mb{P})(\mb{P}y_1)^2\|^2+\frac{t}{4}\|\Gamma_3\|^2\\
\leq&t\|\mb{P}y_1\|^2_{L^{\infty}}\|\hat{\mb{y}}'-\mb{P}\mb{y}\|^2+t\|
(\mb{I}-\mb{P})(\mb{P}y_1\mb{P}y_3)\|^2+\frac{t}{2}\|\Gamma_1\|^2\\
\end{split}
\end{equation}
\begin{equation*}%\label{I_3}
\begin{split}
&+t\|(\mb{I}-\mb{P})(\mb{P}y_1)^2\|^2+\frac{t}{2}\|\Gamma_3\|^2.
\end{split}
\end{equation*}
For $\hat{y}_2'-\mb{P}y_2$, and $\hat{y}_3'-\mb{P}y_3$, we have
\begin{equation*}
\begin{split}
\Big(\hat{y}_2'-\mb{P}y_2, \frac{d}{dt}(\hat{y}_2'-\mb{P}y_2)\Big) =& \Big(\hat{y}_2'-\mb{P}y_2, -\mb{P}\mb{B}(\hat{y}_2',\hat{y}_3')+\mb{P}\mb{B}(y_2,y_3)\Big) \\
&-\Big(\hat{y}_2', t\mb{P}\{\mb{B}(\Gamma_2,\hat{y}_3')+\mb{B}(\hat{y}_2',\Gamma_3)\}\Big)\\
&+\Big(\mb{P}y_2, t\mb{P}\{\mb{B}(\Gamma_2,\hat{y}_3')+\mb{B}(\hat{y}_2',\Gamma_3)\}\Big)\\
\triangleq& J_1+J_2+J_3,\\
%\end{split}
%\end{equation*}
%\begin{equation*}
%\begin{split}
\Big(\hat{y}_3'-\mb{P}y_3, \frac{d}{dt}(\hat{y}_3'-\mb{P}y_3)\Big) =& \Big(\hat{y}_3'-\mb{P}y_3, -\mb{P}\mb{B}(\hat{y}_1',\hat{y}_1')+\mb{P}\mb{B}(\hat{y}_2',\hat{y}_2')\\
&+\mb{P}\mb{B}(y_1,y_1)-\mb{P}\mb{B}(\hat{y}_2',\hat{y}_2')\Big) \\
&-2\Big(\hat{y}_1', t\mb{P}\{\mb{B}(\Gamma_1,\hat{y}_1')\}\Big)+2\Big(\hat{y}_2', t\mb{P}\{\mb{B}(\Gamma_2,\hat{y}_2')\}\Big)\\
&+2\Big(\mb{P}y_1,  t\mb{P}\{\mb{B}(\Gamma_1,\hat{y}_1')\}\Big)-2\Big(\mb{P}y_2,  t\mb{P}\{\mb{B}(\Gamma_2,\hat{y}_2')\}\Big)\\
\triangleq & K_1+K_2+K_3.
\end{split}
\end{equation*}
Following the same steps as in \eqref{I_1} and \eqref{I_3}, we obtain
\begin{equation}
\begin{split}
J_1\leq & C_1 \|\hat{\mb{y}}'-\mb{Py}\|^2+C_1\|\hat{\mb{y}}'-\mb{Py}\|\|(\mb{I}-\mb{P})\mb{y}\|;\\
J_3 \leq & t\|\mb{P}y_2\|^2_{L^{\infty}}\|\hat{\mb{y}}-\mb{P}\mb{y}\|^2+t\|
(\mb{I}-\mb{P})(\mb{P}y_2\mb{P}y_3)\|^2+\frac{t}{2}\|\Gamma_2\|^2\\
&+t\|(\mb{I}-\mb{P})(\mb{P}y_2)^2\|^2+\frac{t}{2}\|\Gamma_3\|^2;\\
K_1\leq & C_2 \|\hat{\mb{y}}'-\mb{Py}\|^2+C_2\|\hat{\mb{y}}'-\mb{Py}\|\|(\mb{I}-\mb{P})\mb{y}\|, \text{for some constant} C_2;\\
K_3 \leq  &4t\|\mb{P}y_3\|^2_{L^{\infty}}\|\hat{\mb{y}}-\mb{P}\mb{y}\|^2+4t\|
(\mb{I}-\mb{P})(\mb{P}y_3\mb{P}y_1)\|^2\\
&+4t\|
(\mb{I}-\mb{P})(\mb{P}y_3\mb{P}y_2)\|^2+\frac{t}{2}\|\Gamma_1\|^2+\frac{t}{2}\|\Gamma_2\|^2.
\end{split}
\end{equation}
Finally, it is easy to verify that $I_2+J_2+K_2 = \frac{1}{2}\frac{d}{dt}\|\hat{\mb{y}}'\|^2 = -t\|\mb{\Gamma}\|^2$. Putting everything together we see that there exist constants $C_3$, $C_4$, $A$, $B$ and $C$ such that
\begin{equation}\label{t-model_er_evolution}
\begin{split}
\frac{1}{2}\frac{d}{dt}\|\hat{\mb{y}}'-\mb{P}\mb{y}\|^2\leq &C_3 \|\hat{\mb{y}}'-\mb{Py}\|^2+C_4\|\hat{\mb{y}}'-\mb{Py}\|\|(\mb{I}-\mb{P})\mb{y}\|\\
&+t\max(\|\mb{P}y_1\|^2_{L^{\infty}},\|\mb{P}y_2\|^2_{L^{\infty}},4\|\mb{P}y_3\|^2_{L^{\infty}})\|\hat{\mb{y}}-\mb{P}\mb{y}\|^2\\
&+5t\|(\mb{I}-\mb{P})(\mb{P}y_1\mb{P}y_3)\|^2+5t\|(\mb{I}-\mb{P})(\mb{P}y_2\mb{P}y_3)\|^2\\
&+t\|(\mb{I}-\mb{P})(\mb{P}y_1)^2\|^2+t\|(\mb{I}-\mb{P})(\mb{P}y_2)^2\|^2\\
\leq &(A+tB)\|\hat{\mb{y}}'-\mb{P}\mb{y}\|^2+C\|(\mb{I}-\mb{P})\mb{y}\|^2\\
&+5t\|(\mb{I}-\mb{P})(\mb{P}y_1\mb{P}y_3)\|^2+5t\|(\mb{I}-\mb{P})(\mb{P}y_2\mb{P}y_3)\|^2\\
&+t\|(\mb{I}-\mb{P})(\mb{P}y_1)^2\|^2+t\|(\mb{I}-\mb{P})(\mb{P}y_2)^2\|^2
\end{split}
\end{equation}
\end{proof}

From \eqref{t-model_er_evolution} we can see that the contribution of the $t$-model term to the error of $t$-model approximation is expressed as
\begin{equation}\label{t-model_er}
\begin{split}
&5t\|(\mb{I}-\mb{P})\mb{B}(\mb{P}y_1,\mb{P}y_3)\|^2+5t\|(\mb{I}-\mb{P})\mb{B}(\mb{P}y_2,\mb{P}y_3)\|^2\\
&+t\|(\mb{I}-\mb{P})\mb{B}(\mb{P}y_1,\mb{P}y_1)\|^2+t\|(\mb{I}-\mb{P})\mb{B}(\mb{P}y_2,\mb{P}y_2)\|^2.
\end{split}
\end{equation}
Compared with
\begin{equation}
\begin{split}
\frac{1}{2}\Big|\frac{d}{dt}\|\hat{\mb{y}}'\|^2\Big|\leq & t\big(\|(\mb{I}-\mb{P})\mb{B}(\hat{y}_1',\hat{y}_3')\|^2+\|(\mb{I}-\mb{P})\mb{B}(\hat{y}_2',\hat{y}_3')\|^2\\
&+2\|(\mb{I}-\mb{P})\mb{B}(\hat{y}_1',\hat{y}_1')\|^2+2\|(\mb{I}-\mb{P})\mb{B}(\hat{y}_2',\hat{y}_2')\|^2\big),
\end{split}
\end{equation}
we can conclude that $\frac{1}{2}\Big|\frac{d}{dt}\|\hat{\mb{y}}'\|^2\Big|$ is a good indicator of the rate of change of the error due to the $t$-model term. Meanwhile the error generated by the $t$-model term signifies the energy moving from the resolved modes to the unresolved modes of the full system. Thus the error can be controlled by controlling $\Big|\frac{d}{dt}\|\hat{\mb{y}}'\|^2\Big|$. In particular, $\Big|\frac{d}{dt}\|\hat{\mb{y}}'\|^2\Big|\leq TOL_1$ provides a good criterion for mesh refinement.

\end{document}